\documentclass[11pt]{amsart}
\usepackage{amssymb, amsfonts, amsmath, amsthm}

\input xy
\usepackage[all]{xy}

\theoremstyle{plain}
\newtheorem{teo}{Theorem}[section]
\newtheorem{prop}[teo]{Proposition}
\newtheorem{cor}[teo]{Corollary}
\newtheorem{lem}[teo]{Lemma}

\theoremstyle{definition}
\newtheorem{defi}[teo]{Definition}

\theoremstyle{remark}
\newtheorem{rem}[teo]{Remark}

\input xy
\usepackage[all]{xy}

\newcommand{\z}  {\mathbf{Z}}
\newcommand{\ci} {\mathbf{C}}
\newcommand{\re} {\mathbf{R}}
\newcommand{\qu} {\mathbf{Q}}

\newcommand{\letr}[1] {\mathcal{1}}

\newcommand{\f } {\mathbf{F}}

\DeclareMathOperator{\coker}{coker}
\DeclareMathOperator{\gl}{GL}

\DeclareMathOperator{\enn}{End}
\DeclareMathOperator{\ho}{Hom}

\DeclareMathOperator{\trace}{tr}

\def\A{\mathbf{A}}

\begin{document}

\title{On certain cusp forms on a definite quaternion}
\author{Tommaso Giorgio Centeleghe}
\address{Universit\"at Heidelberg\\
IWR, Im Neuenheimer Feld 368\\
69120 Heidelberg, Germany\\}
\email{tommaso.centeleghe@iwr.uni-heidelberg.de }

\maketitle

\begin{abstract} If $D$ is the definite quaternion over $\qu$ of discriminant $p$, we compute,
for any prime $p>3$, the number of infinite dimensional cusp forms on $D^*$ which are trivial
at infinity, tamely ramified at $p$, and have given conductor $N$ away from $p$. We include
a detail explanation of a Deuring--type correspondence between supersingular elliptic curves
in characteristic $p$ and a certain double coset arising from the adelic points of $D^*$.
\end{abstract}

\section*{Introduction}

Let $p$ be a prime number, and $D$ the quaternion algebra over $\qu$ ramified
precisely at $p$ and infinity. Let $\pi$ be an irreducible, unitary, automorphic representation
of the multiplicative group $D^*$ of $D$.
For a prime number $\ell$, let $\pi_\ell$ be the local component of $\pi$ at $\ell$, and let
$\pi_\infty$ be the component of $\pi$ at the archimedean
place of $\qu$. The representations $\pi_p$ and
$\pi_\infty$, of the respective local groups ${(D\otimes\qu_p)}^*$ and ${(D\otimes\re)}^*$, are both
finite dimensional. In this paper we will only be concerned with those $\pi$ so that

i) $\pi_\infty$ is the trivial, one--dimensional representation;

ii) $\pi_p$ has a nonzero vector fixed by the maximal pro--$p$ subgroup of $D^*_p$.
\\Condition ii) can be interpreted as a tame ramification requirement for $\pi$ at $p$.
For a given integer $N\geq 1$ not divisible by $p$, let $\mathcal{A}(p,N)$
denote the set of isomorphism classes of automorphic representations $\pi$ of $D^*$ of the type
described above and whose prime--to--$p$ conductor is equal to $N$. In this paper we obtain an
explicit formula for the cardinality $A(p,N)$ of $\mathcal{A}(p,N)$, and of some distinguished subsets (cf. Thm. \ref{MT}).
For simplicity $p$ is assumed to be $>3$.

Associated to {\it any} irreducible, unitary, automorphic representation $\pi$ of $D^*$ there is a
system of Hecke eigenvalues $\Phi_\pi={(a_\ell)}_{\ell\nmid pN}$, where $\ell$ is a prime number not
dividing $pN$, and $N$ is an integer divisible by the prime--to--$p$ part of the conductor of $\pi$.
When the central character $\chi_\pi$ of $\pi$ has finite order (always the case if $\pi$ is trivial at
infinity), then it is known that the $a_\ell$ are algebraic integers of a certain number field $K(\pi)$,
and therefore, once a prime of $\bar\qu$ above $p$ is chosen, can be reduced mod $p$
to obtain $\overline\Phi(\pi)={(\bar a_\ell)}_{\ell\nmid pN}$, a system of Hecke eigenvalues mod $p$.

The tameness condition imposed to the automorphic forms $\pi$ is natural when studying the reduction mod $p$
of the totality of the systems of eigenvalues arising from cusp forms on $D^*$ trivial at infinity. A theorem of Serre relates
these mod $p$ systems of eigenvalues to those arising from mod $p$ modular forms. In a forthcoming work, the author
will adopt Serre's quaternionic viewpoint to estimate from above, using the formulas of this paper, the number of mod $p$
Hecke eigenforms of given conductor.

The author would like to express his gratitude to professor Chandrashekhar Khare for suggesting to him
to perform these ``quaternionic counts", as well as for the assistance he has received. This work benefitted
from many discussions with professors Gordan Savin and Gebhard B\"ockle. The author heartily thanks
them. In particular, it was professor B\"ockle who suggested to the author that it could have been possible
to find an exact formula for $A(p,N)$.

\section{Statement of the Theorem}

\subsection{Generalities}

Let $p$ be a prime number, and $D$ the quaternion algebra over $\qu$ ramified precisely at $p$
and infinity. The multiplicative group $D^*$ of $D$ is an algebraic group over $\qu$ that can be realized
as a closed subgroup of $\gl_4$. For any commutative $\qu$--algebra $K$, the group of $K$--valued points of
$D^*$ is denoted by $D^*_K$. If $K$ is also a topological algebra, then the embedding
$D^*_K\subset\gl_4(K)$ makes $D^*_K$ a topological group. If $\nu\in\Sigma_\qu$ is any place of $\qu$, and $\qu_\nu$ is
the corresponding completion, we will write $D^*_\nu$ for $D^*_{\qu_\nu}$; when no confusion can arise
we may simply write $D^*$ for $D^*_\qu$. A finite place $\nu\in\Sigma_\qu$ will often be denoted by
its residual characteristic $\ell$.

The group $D^*_\qu$ sits inside the adelic group $D^*_\A$ as a discrete subgroup,
the homogeneous space $X=D^*_\qu\backslash D^*_\A$ is locally compact and admits a measure $dx$ that
is invariant under the right translation action of $D^*_\A$. The center $Z$ of $D^*$, viewed as a closed,
algebraic subgroup, is isomorphic to the multiplicative group $\qu^*$, if $\psi:Z_\A\rightarrow\ci^*$ is
a continuous Hecke character of finite order then the space $L^2(X,\psi)$ is defined to be the space
of measurable functions $\varphi:D^*_\A\rightarrow\ci$ such that:

i) $\varphi(\gamma g)=\varphi(g)$, for all $\gamma\in D^*_\qu$, $g\in D^*_\A$;

ii) $\varphi(gz)=\psi(z)\varphi(g)$, for all $z\in Z_\A$, $g\in D^*_\A$;

iii) $\int_{X/Z_\A}{|f(x)|}^2dx<\infty$.
\\The space $L^2(X,\psi)$ is a Hilbert space on which $D^*_\A$ acts unitarily by right translation. The following
fundamental theorem is proved in \cite{GG}:

\begin{teo} The unitary representation $\rho_\psi$ of $D^*_\A$ on $L^2(X,\psi)$ decomposes uniquely as the Hilbert
direct sum
$$\rho_\psi=\bigoplus_{\pi\in I_\psi}' m(\pi)\pi$$
of a countable collection $I_\psi$ of irreducible, pairwise non--isomorphic, subrepresentations, each appearing
with finite multiplicity.
\end{teo}

The next result is due to Jacquet--Langlands (cf. \cite{JL}, Prop. 11.1.1):

\begin{teo} The multiplicity $m(\pi)$ of each constituent $\pi$ appearing in $\rho_\psi$ is one.
\end{teo}

Any irreducible representation $\pi$ occurring in the above decomposition of $\rho_\psi$, for some $\psi$, will be
referred to as a {\it cusp form}.

In virtue of the tensor product Theorem (cf. \cite{GG}), any cusp form $\pi$ is isomorphic to a restricted tensor product of a unique
family ${(\pi_\nu)}_{\nu\in\Sigma_\qu}$ of irreducible, unitary representations of the local groups $D^*_\nu$. For any prime
number $\ell\neq p$, there is an isomorphism $D^*_\ell\simeq\gl_2(\qu_\ell)$, and for almost all $\ell$ the local
component $\pi_\ell$ of $\pi$ defines a representation of $\gl_2(\qu_\ell)$ that is unramified. If $\pi$
is infinite dimensional, then it follows from the approximation theorem that for every prime $\nu$ that is unramified in
$D$ the local component $\pi_\nu$ is also infinite dimensional. On the other hand, $\pi_\nu$ is finite dimensional
when $\nu$ is equal to $p$ or $\infty$.

If $\pi$ is infinite dimensional and $\ell\neq p$, then the conductor $c(\pi_\ell)$ is defined (cf. \cite{Del}, \S $2.2$, or \cite{Ca}), and it will be thought of as a
positive integer given by the appropriate power of $\ell$, instead that of an ideal of $\z$. By definition, the prime--to--$p$
conductor $N(\pi)$ of $\pi$ is
$$N(\pi)=\prod_{\ell\neq p}c(\pi_\ell),$$
it is a well defined integer since $\pi_\ell$ is unramified for almost all $\ell$.

We will be concerned only with those {\it infinite} dimensional cusp forms $\pi$ such that

i) $\pi_\infty$ is the trivial, one--dimensional representation;

ii) $\pi_p$ has a nonzero vector fixed by the maximal pro--$p$ subgroup of $D^*_p$.
\\In the next section we digress on the consequences of condition ii)
on the nature of the unitary, local representations $\pi_p$ that may occur in the decomposition of $\pi$.

\subsection{The local nature at $p$}\label{lnp}

The algebra $D_p=D\otimes\qu_p$ is a quaternion, division algebra over $\qu_p$, it is unique up to isomorphism.
On $D_p$ a valuation function is defined and the corresponding valuation ring $O_p$ is the unique maximal,
compact subring. It consists of all the elements that are integral over $\z_p$. There is a unique maximal, two--sided ideal
$\mathfrak{m}$ of $O_p$, which is principal and generated by any {\it uniformizer} $\varpi\in\mathfrak{m}$. The residue
field $k$ is a degree $2$ extension of $\f_p$, conjugation by any uniformizer $x\rightarrow \omega x\omega^{-1}$ preserves
$O_p$ and $\mathfrak{m}$, and induces the Frobenius automorphism on $k$. The maximal pro--$p$ subgroup of $O_p^*$
is the left term of the exact sequence $$1\rightarrow 1+\mathfrak{m}\longrightarrow O_p^*\longrightarrow k^*\rightarrow 1,$$
and will be denoted by $O_p^*(1)$. 

Let now $\pi_p$ be an irreducible, unitary representation of $D_p^*$ on a finite dimensional complex vector space $V$.
The subspace $V'$ of elements fixed by $O_p^*(1)$ is stable under $D^*_p$. It follows that if $\pi_p$ satisfies
ii) then, by its irreducibility, we have $V'=V$. The group $k^*$ acts on $V'=V$ and there is a decomposition
$$V=\bigoplus_{\eta\in\hat k^*}V^{\eta}$$
into isotypical components indexed by complex valued characters of $k^*$. For every such $\eta$, the action of a uniformizer $\varpi$
on $V$ induces an isomorphism of the summand $V^{\eta}$ with $V^{\eta^p}$, its Frobenius conjugate. The irreducibility of
$\pi_p$ forces the following two possibilities:

1) $\dim V=1$ and $V=V^{\eta}$, for a character $\eta$ {\it equal} to its Frobenius conjugate $\eta^p$;

2) $\dim V=2$, and $V=V^{\eta}\oplus V^{\eta^p}$, with $\dim V^{\eta}=\dim V^{\eta^p}=1$, for a pair of {\it distinct}
characters $(\eta,\eta^p)$.

In the first case $\pi_p$ is abelian and the character describing the action factors
through the reduced norm map $N_p:D_p^*\rightarrow\qu_p^*$ and can therefore be identified with
a certain character $\epsilon_\eta$ of $\qu_p^*$ that is tamely ramified, i.e., it is trivial
on the units $\z_p^*(1)$ that are congruent to $1$ mod $p$.

In the second case, if $\varpi\in D$ is the uniformizer such that $\varpi^2=p$, then $\varpi$ acts on $V$ interchanging the two
lines $V^{\eta}$ and $V^{\eta^p}$. Its square acts as multiplication by $\chi_{\pi_p}(p)$, where $\chi_{\pi_p}$ is
the central character of $\pi_p$.

\subsection{Statement of the Theorem}\label{statementteo}

For any integer $N\geq 1$ not divisible by $p$, we let $\mathcal{A}(p,N)$ denote the set of
infinite dimensional, unitary cusp forms $\pi$ on $D^*$ occurring in $L^2(X,\psi)$, for some
finite order character $\psi$, satisfying i) and ii), and such that $N(\pi)=N$.

Let $\mathcal{A}_1(p,N)$ (resp. $\mathcal{A}_2(p,N)$) be the subset of $\mathcal{A}(p,N)$
given by those cusp forms $\pi$ for which $\pi_p$ is one dimensional (resp. two dimensional),
and denote its cardinality by $A_1(p,N)$ (resp. $A_2(p,N)$). Moreover let $\mathcal{A}_0(p,N)$ be the
subset of $\mathcal{A}_1(p,N)$ given by the cusp forms $\pi$ so that $\pi_p$ is trivial on $O_p^*$, i.e., in
the notation used in subsection \ref{lnp}, so that the associated character $\epsilon_\eta:\qu_p^*\rightarrow\ci^*$
is {\it unramified}. Let $A_0(p,N)$ denote be its cardinality. Theorem \ref{MT} below provides, for $p>3$,
closed formulas for $A_0(p,N)$, $A_1(p,N)$ and $A_2(p,N)$.

For an integer $N\geq 1$ not divisible by $p$, let $\Delta(p,N)$ be the function defined by the following table.
Its value depends on $N$ and on the congruence class of $p$ modulo $12$.

\begin{tabular}{c|cccc}
$p$ mod $12$&$1$&$5$&7&$11$\\
\hline
$N=1$&$\dfrac{{(p-1)}^2}{24}$&$\dfrac{(p-1)(p+15)}{24}$&$\dfrac{(p-1)(p+11)}{24}$&$\dfrac{(p-1)(p+27)}{24}$\\
$N=2$&$\dfrac{{(p-1)}^2}{8}$&$\dfrac{{(p-1)}^2}{8}$&$\dfrac{(p-1)(p+3)}{8}$&$\dfrac{(p-1)(p+3)}{8}$\\
$N=3$&$0$&$\dfrac{2(p-1)}{3}$&$0$&$\dfrac{2(p-1)}{3}$\\
$N>3$&0&0&0&0\\
\end{tabular}

If $f,g:\z_{>0}\rightarrow\qu$ are functions defined over the positive integer and valued in $\qu$,
then $(f*g)$ will denote their convolution. 
We say that $f$ is multiplicative if $f(1)=1$ and $f(mn)=f(m)f(n)$ for all pairs of coprime positive integers
$m$ and $n$. The convolution product is associative, commutative and $(f*g)$ is multiplicative
when $f$ and $g$ are. The M\"obius function $\mu$ is the multiplicative function vanishing on
integers that are not square free, and taking value $-1$ on every prime. For a prime $\ell$, and an
integer $n\geq 1$, let $r$ be multiplicative function defined by

$r(\ell^n)=\left\{\begin{matrix}
\ell^2-3&{\rm if }\hphantom{x}n=1;\\
\ell^4-3\ell^2+3&{\rm if }\hphantom{x}n=2;\\
\ell^{2(n-3)}{(\ell^2-1)}^3&{\rm if }\hphantom{x}n>2;\\
\end{matrix}\right.$

\begin{teo}\label{MT} Let $p>3$ be a prime number. Then

$A_0(p,N)=\dfrac{r(N)(p-1)}{24}+\dfrac{(\Delta(p,\_)*\mu*\mu)(N)}{p-1}-\mu(N);$

$A_1(p,N)=\dfrac{r(N){(p-1)}^2}{24}+(\Delta(p,\_)*\mu*\mu)(N)-(p-1)\mu(N);$

$A_2(p,N)=\dfrac{r(N){p(p-1)}^2}{48}-\dfrac{(\Delta(p,\_)*\mu*\mu)(N)}{2}.$
\end{teo}

The convolution products in the theorem are performed with respect to the second variable of $\Delta$.

\begin{cor} Let $p>3$ be a prime number, and $N\geq 1$ an integer not divisible by $p$. The dimension
of the space of cuspidal newforms of weight $2$ on the group $\Gamma_1(pN)$, with trivial character
locally at $p$ is equal to
$$A_0(p,N)=\dfrac{r(N)(p-1)}{24}+\dfrac{(\Delta(p,\_)*\mu*\mu)(N)}{p-1}-\mu(N).$$
\end{cor}
\begin{proof} From the Jacquet--Langlands global correspondence (cf. \cite{JL}, or \cite{Ro}) it follows that cusp forms
of $\mathcal{A}_0(p,N)$ are in bijection with cusp forms $\pi$ on $\gl_2$ of conductor $pN$, with
archimedean component $\pi_\infty$ isomorphic to the principal series of lowest weight $2$,
and such that $\pi_p$ is of Steinberg type at $p$ with unramified central character.

Since any form of conductor $pN$, and with unramified central character at $p$ has to be
of Steinberg type at $p$, the formula follows.
\end{proof}

\section{Automorphic forms on $D^*$}\label{autom}

We introduce here certain spaces $S(1,N)$ of automorphic forms that intervene in the study of the cusp forms introduced
in the previous section. They consist of locally constant functions on $D^*_\qu\backslash D^*_\A$ and they are independent on the
archimedean variable $g_\infty\in D^*_\infty$.

Let $R$ be a maximal order of $D$, for a prime number $\ell$ we set

$R_\ell=R\otimes\z_\ell$;

$D_\ell=D\otimes\qu_\ell=R_\ell\otimes\qu$;
\\where all tensor products are taken over $\z$. If $\ell\neq p$ the ring $D_\ell$ is isomorphic to the algebra $\textrm{M}_2(\qu_\ell)$
of two--by--two matrices with entries in $\qu_\ell$, and we fix an identification $D_\ell\simeq\textrm{M}_2(\qu_\ell)$ such that
$R_\ell$ corresponds to the standard maximal $\z_\ell$--order $\textrm{M}_2(\z_\ell)$.

Let $N$ be any integer $\geq 1$ not divisible by $p$, using the identification $R^*_\ell\simeq\textrm{GL}_2(\z_\ell)$ we define a
congruence subgroup $U_\ell(N)$ of $R^*_\ell$ as follows
$$U_\ell(N)=\left\{x\in\textrm{GL}_2(\z_\ell)|x\equiv
\begin{pmatrix}
*&*\\
0&1\\
\end{pmatrix}\textrm{mod}\hphantom{x}\textrm{N}\right\}.$$
We have that $U_\ell=\textrm{GL}_2(\z_\ell)$ if and only if $\ell$ does not divide $N$. The normalizer $U'_\ell(N)$ of
$U_\ell(N)$ in $R^*_\ell$ is
$$U'_\ell(N)=\left\{x\in\textrm{GL}_2(\z_\ell)|x\equiv
\begin{pmatrix}
*&*\\
0&*\\
\end{pmatrix}\textrm{mod}\hphantom{x}\textrm{N}\right\},$$
the quotient $U_\ell(N)/U'_\ell(N)$ is isomorphic to ${(\z_\ell/N\z_\ell)}^*$.

The ring $R_p$ is the unique maximal compact subring of $D_p$, its residue field will be denoted by $k$, and
the maximal pro--$p$ subgroup of its multiplicative group by $R_p^*(1)$ (cf. section \ref{lnp}).

Let $U(1,N)$ be the open subgroup of $D^*_\A$ defined by
$$U(1,N)=\prod_{\ell\neq p}U_\ell(N)\times R_p^*(1)\times D^*_\infty,$$
and set
\begin{equation}\label{doublecoset}\Omega(1,N)=D^*_\qu\backslash D^*_\A/U(1,N).
\end{equation}
This double coset is finite and discrete, since the space $D^*_\qu\backslash D^*_\A/D^*_\infty$ is compact.

\begin{defi} The space $S(1,N)$ of automorphic forms on $D^*$ of level $(1,N)$ and trivial at infinity is
the space of functions $f:\Omega(1,N)\rightarrow\ci$.
\end{defi}

Elements of $S(1,N)$ are therefore complex valued functions on $D^*_\A$ that are invariant under
translation by the discrete subgroup $D^*_\qu$ to the left, and under translation by the open subgroup
$U(1,N)$ to the right.

There are several operators on $S(1,N)$ that we now describe. For any prime $\ell\neq p$ the $\ell$--th
Hecke operator $T_\ell$ is defined as follow. The double coset
$U_\ell(N)\begin{pmatrix}
\ell&0\\
0&1\\
\end{pmatrix}U_\ell(N)$, considered as a subset of $\gl_2(\qu_\ell)$, is the finite
union of disjoint left cosets
\begin{equation}\label{cosetdec}U_\ell(N)\begin{pmatrix}
\ell&0\\
0&1\\
\end{pmatrix}U_\ell(N)=\bigsqcup_{i}\alpha_iU_\ell(N).
\end{equation}
The number of left cosets in the above decomposition is $\ell+1$ if $\ell\nmid pN$, and $\ell$ if $\ell| N$.
If $f\in S(1,N)$, then $T_\ell(f)$ is defined by the formula
\begin{equation}\label{defHecke} T_\ell(f)(x)=\sum_if(x\alpha_i),
\end{equation}
where $\alpha_i\in\gl_2(\qu_\ell)\simeq D^*_\ell$ is considered an element of $D^*_\A$ thank
to the natural embedding $D^*_\ell\subset D^*_\A$. It is clear from their definition that the operators
$T_\ell$ commute with each other.

In section \ref{isoclass} the set $\Omega(1,N)$ will be given a moduli interpretation in terms of supersingular
elliptic curves. The Hecke operator $T_\ell$ corresponds to a certain averaging operator over degree $\ell$ isogenies (cf. section \ref{operiamo}).
It will be shown that:

\begin{lem}\label{semisimple} If $\ell\nmid pN$ then $T_\ell$ is a {\it semisimple} endomorphism of $S(1,N)$.
\end{lem}

A consequence of the lemma is that $S(1,N)$ decomposes into the direct sum of common eigenspaces for the
$T_\ell$, where $\ell\nmid pN$.

The normalizer of $U(1,N)$ in $D^*_\A$ is
$$U'(1,N)=\prod_{\ell\neq p}U'_\ell(N)\times D^*_p\times D^*_\infty,$$
it acts naturally on $S(1,N)$ by right translation, and an action of
$$U'(1,N)/U(1,N)={(\z/N\z)}^*\times D_p^*/R^*_p(1)$$
on $S(1,N)$ can be deduced. The automorphism of $S(1,N)$ determined by $d\in{(\z/N\z)}^*$ is
denoted by $\langle d\rangle$ and called {\it diamond operator}.

The isomorphism class of the $k^*$ representation on $S(1,N)$ deduced from the action of
$D_p^*/R_p^*(1)$ will be determined in section \ref{isoclasskstar}. We will see that $S(1,N)$ is close to be a
multiple of the regular representation for $k^*$.

We conclude the section describing an Hecke invariant subspace ${S(1,N)}^{\rm Eis}$ of $S(1,N)$.

\begin{defi} The subspace of $S(1,N)$ given by those functions $f:D^*_\A\rightarrow\ci$ that can be
factored through the reduced norm map ${\rm Nr}:D^*_\A\rightarrow \A^*$ is denoted by
${S(1,N)}^{\rm Eis}$.
\end{defi}

Fix now a character $\chi:\f_p^*\rightarrow\ci^*$ of order $p-1$.

\begin{prop}\label{eisen} The space ${S(1,N)}^{\rm Eis}$ has dimension $(p-1)$ and is independent of $N$.
It has a basis $(e_1,\ldots,e_{p-1})$ consisting of eigenvectors for all the Hecke operators $T_\ell$,
with $\ell\neq p$, such that if $\ell\nmid pN$ we have $T_\ell(e_i)=(\ell + 1)\chi^{-i}(\ell)e_i$.
\end{prop}
The meaning of the first assertion is that ${S(1,N)}^{\rm Eis}$ defines a space of locally constant
functions on $D^*_\A$ which has dimension $(p-1)$ and that is independent on $N$.
\begin{proof} Let
$$\A^*=\qu^*\times\hat\z^*\times\re^*_{>0}$$
be the canonical decomposition of $\A^*$ into the product of the discrete subgroup $\qu^*$ and the open
subgroup $\hat\z^*\times\re^*_{>0}$. The image of the reduced norm map ${\rm Nr}:D^*_\A\rightarrow\A^*$
is the index two subgroup which, in the above decomposition of $\A^*$, corresponds to
$\qu^*_{>0}\times\hat\z^*\times\re^*_{>0}$. The space ${S(1,N)}^{\rm Eis}$ is identified with the space of functions
$$f:\qu^*_{>0}\times\hat\z^*\times\re^*_{>0}\longrightarrow\ci$$
that are invariant under the image of $D^*_\qu\cdot U(1,N)$ with respect to the reduced
norm map ${\rm Nr}$. Since
$${\rm Nr}(D^*_\qu\cdot U(1,N))=\qu^*_{>0}\times\prod_{\ell\neq p}\z_\ell^*\times\z_p^*(1)\times\re^*_{>0},$$
where $\z_p^*(1)$ is the subgroup of $\z_p^*$ of units that are congruent to $1$ modulo $p$, the elements of
${S(1,N)}^{\rm Eis}$ are the functions on $D^*_\A$ that factor through the composition
$$r\cdot{\rm Nr}:D^*_\A\longrightarrow\qu^*_{>0}\times\hat\z^*\times\re^*_{>0}\longrightarrow\f_p^*,$$
where the fist map {\rm Nr} is the reduced norm, and $r$ is the projection onto the quotient group
$$\left(\qu^*_{>0}\times\hat\z^*\times\re^*_{>0}\right)/\left(\qu^*_{>0}\times\prod_{\ell\neq p}\z_\ell^*\times\z_p^*(1)\times\re^*_{>0}\right)=\f_p^*.$$
The first part of the proposition then follows.

To complete the proof, we construct an explicit basis for ${S(1,N)}^{\rm Eis}$ given by eigenvectors for all the Hecke operators $T_\ell$, with
$\ell\neq p$. For an integer $j$ with $1\leq j\leq p-1$, define $e_j\in {S(1,N)}^{\rm Eis}$ to be the composition $\chi^j\cdot r\cdot{\rm Nr}$.
By the linear independence of characters, we have that $(e_1,\ldots,e_{p-1})$ is a basis of ${S(1,N)}^{\rm Eis}$. Let now $\ell\nmid pN$
be a prime number, and $x\in D^*_\A$ any element. By definition of $T_\ell$, we have
$$(T_\ell e_j)(x)=\sum_{i}e_j(x\alpha_i),$$
where the $\alpha_i\in D^*_\ell\subset D^*_\A$ are representatives for the double coset in formula $(\ref{cosetdec})$.
Observe now that $e_j(x\alpha_i)=e_j(x)e_j(\alpha_i)$, moreover $e_j(\alpha_i)=\chi^{j}\cdot r\cdot{\rm Nr}(\alpha_i)$.
For any $\alpha_i$, the idele ${\rm Nr}(\alpha_i)$ is $1$ at every place other than the $\ell$--adic one, and it is equal to
$\ell$ at the $\ell$--adic place. It is easy to see that $r\cdot{\rm Nr}(\alpha_i)=\ell^{-1}\in\f_p^*$, therefore
$e_j(\alpha_i)=\chi^{j}(\ell^{-1})$ and the proposition follows.
\end{proof}

\section{Hecke eigenvalues and cusp forms}

We recall the details of the dictionary between systems of Hecke eigenvalues arising from the module
$S(1,N)$ introduced in the previous section, and unitary cusp form on $D^*$ satisfying conditions i) and ii).
In this correspondence, the strong multiplicity one result for $L^2(X,\psi)$ plays an important role.

\begin{defi} A system of Hecke eigenvalues arising from $S(1,N)$ is a collection $\Phi={(a_\ell)}_{\ell\nmid pN}$ of
complex numbers such that there exists a nonzero element $f\in S(1,N)$ with $T_\ell (f)=a_\ell f$ for all primes $\ell\nmid pN$.
\end{defi}
For a system of eigenvalues $\Phi$ we define the corresponding isotypical component of $S(1,N)$ as
$${S(1,N)}^{\Phi}=\{f\in S(1,N)\hphantom{.}|\hphantom{.}T_\ell(f)=a_\ell
f,\hphantom{.} {\rm for}\hphantom{.}{\rm all}\hphantom{.}\ell\nmid pN\}.$$
According to lemma \ref{semisimple}, there is a decomposition
$$S(1,N)=\bigoplus_{\Phi}{S(1,N)}^{\Phi}$$
of $S(1,N)$ into the direct sum of its isotypical components. Since the action on $S(1,N)$ by right translation of the normalizer
$U'(1,N)$ of $U(1,N)$ commutes with that of $T_\ell$, for all primes $\ell\nmid pN$, each summand ${S(1,N)}^\Phi$ is
a representation of $U'(1,N)/U(1,N)={(\z/N\z)}^*\times D^*_p/R^*_p(1)$.

Let now $\pi$ be an element of $\mathcal{A}(p,N)$; that is, $\pi$ is an infinite dimensional, unitary cusp form on $D^*$ of conductor $N$,
and satisfying conditions i) and ii), which occurs on a closed subspace $V(\pi)$ of $L^2(X,\psi)$, for some central character $\psi$ whose
conductor divides $N$. Associated to $\pi$ there is a system of Hecke eigenvalues $\Phi={(a_\ell)}_{\ell\nmid pN}$:

\begin{prop} The space $V(\pi)\cap S(1,N)$ is non empty. For every prime $\ell\nmid pN$ the operator $T_\ell$ acts on
on $V(\pi)\cap S(1,N)$ via multiplication by a certain scalar $a_\ell$.
\end{prop}
By definition, the collection $\Phi(\pi)={(a_\ell)}_{\ell\nmid pN}$ is the system of eigenvalues associated to $\pi$.
\begin{proof} The space $V(\pi)$ is isomorphic to a restricted tensor product
$\otimes'V(\pi_\ell)$ of local representations, where we may disregard the trivial component at infinity. It follows that
the space ${V(\pi)}^{U(1,N)}$ of $U(1,N)$--invariant vectors is the restricted tensor product
$$\otimes'_{\ell\neq p}{V(\pi_\ell)}^{U_\ell(N)}\otimes{V(\pi_p)}^{R_p^*(1)}.$$
For any $\ell\nmid pN$ the representation $\pi_\ell$ is unramified, and the space ${V(\pi_\ell)}^{U_\ell(N)}$ is one--dimensional
(cf. \cite{Del}, \S 2.2). The Hecke operator $T_\ell$ acts on it as scalar multiplication, the proposition follows.
\end{proof}

A consequence of the strong multiplicity one result for $L^2(X,\psi)$, as we shall see below, is that in fact the space
$V(\pi)\cap S(1,N)$ {\it coincide} with the full isotypical component ${S(1,N)}^{\Phi(\pi)}$.

Conversely, to any system of eigenvalues $\Phi$ occurring in $S(1,N)$ one can attach a cusp form $\pi(\Phi)$ on $D^*$
as follow. Let $f\in {S(1,N)}^\Phi$ be any nonzero automorphic form, and set
$$L^2(X,N)=\bigoplus_{\psi\in\widehat {{(\z/N\z)}^*}}L^2(X,\psi).$$
Consider the smallest closed subspace $V_f$ of $L^2(X,N)$ that contains $f$ and is stable under $D^*_\A$, and denote
by $\pi_f$ the corresponding unitary representation of $D^*_\A$. The following proposition is a standard
consequence of the strong multiplicity one result for $L^2(X,N)$ (cf. \cite{Ca}, Thm. 2):

\begin{prop}\label{eigcus} The representation $\pi_f$ is irreducible and defines a cusp form on $D^*$. The space ${V_f}^{U(1,N)}$ of
$U(1,N)$--invariant vectors coincide with ${S(1,N)}^{\Phi}$, and $V_f$ depends only on $\Phi$, and not on the choice of
$0\neq f\in S(1,N)^{\Phi}$. The action of ${(\z/N\z)}^*$ on ${S(1,N)}^{\Phi}$ is via homotheties specified by a character
$\psi_0:{(\z/N\z)}^*\rightarrow\ci^*$. The action of $D^*_p/R^*_p(1)$ on ${S(1,N)}^{\Phi}$ is isomorphic to a multiple of ${(\pi_{f})}_p$,
where ${(\pi_f)}_p$ is the local component at $p$ of $\pi_f$, moreover $p\in D_p^*$ acts on ${S(1,N)}^\Phi$ as
multiplication by ${\psi_0(p)}^{-1}$.
\end{prop}

The cusp form $\pi_f$ and the Hilbert space $V_f$ will also be denoted by $\pi_\Phi$ and $V_\Phi$, respectively.
The central character $\psi_\Phi$ of $\pi_\Phi$ is obtained by composing the natural map
$$\A^*\longrightarrow\hat\z^*\longrightarrow{(\z/N\z)}^*$$
with the character $\psi_0$ given by the Proposition \ref{eigcus}. We have that $\pi_\Phi$ is an irreducible constituent of $L^2(X,\psi_\Phi)$.

The next lemma characterizes the automorphic functions $f\in S(1,N)$ so that $\pi_f$ is finite dimensional:

\begin{lem}\label{findimonedim} Let $f\in S(1,N)$ be an eigenvector for all the Hecke operator $T_\ell$, with $\ell\nmid pN$. If $\pi_f$
is finite dimensional, then it is one--dimensional, and $f\in {S(1,N)}^{\rm Eis}$.
\end{lem}
\begin{proof} Let $g\in D^*\subset D^*_\A$ be any element, write $g=g_pg^{(p)}$, where $g_p$ is the adelic element of
$D^*_\A$ whose only non--trivial component is the $p$--th component of $g$ in $D_p^*$. Then we have that
\begin{equation}\label{onedimeatp}\pi_f(g_p)={\pi_f(g^{(p)})}^{-1},
\end{equation}
since $\pi_f(D^*)$ is trivial. By assumption, for any $\ell\neq p$ the local representation ${(\pi_f)}_\ell$ is described by
a character, therefore equation (\ref{onedimeatp}) implies that the restriction to $D^*$ of the local representation
${(\pi_f)}_p$ is given described by a character. Now since ${(\pi_f)}_p$ is trivial on an open subgroup of $D_p^*$ and
is described by a character on the dense subgroup $D^*$, we have that ${(\pi_f)}_p$ is abelian and hence one--dimensional.
It follows that $f$ factors through the reduced norm map ${\rm Nr}:D^*_\A\rightarrow\A^*$.
\end{proof}

Together with Proposition \ref{eisen}, Lemma \ref{findimonedim} implies that there are precisely $(p-1)$ systems of eigenvalues $\Phi$
arising from $S(1,N)$ and for which $\pi_\Phi$ is finite dimensional. Such $\Phi$ have been explicitly described 
in Proposition \ref{eisen}.

Let now $\Phi={(a_\ell)}_{\ell\nmid pN}$ be a system of eigenvalues arising from $S(1,N)$ such that the associated
cusp form $\pi=\pi_\Phi$ is {\it infinite} dimensional. All the local components $\pi_\ell$ for $\ell\neq p$ are also infinite
dimensional, and the prime to $p$ conductor of $\pi$ is denote by $N(\pi)$.

Let $t:\z_{>0}\rightarrow\z_{>0}$ be the function whose value in $n>0$ is the number of all positive divisors of $n$.

\begin{lem}\label{locallem} The level $N$ is divisible by $N(\pi)$, moreover
$$\dim({S(1,N)}^\Phi)=\dim(\pi_p)\cdot t(N/N(\pi)).$$
\end{lem}

The lemma follows from a classical local result of Casselman (cf. \cite{Del} \S 2.2) and is the crucial ingredient that enables us to obtain
a recursive formulas that the functions $A_i(p,N)$ in the main theorem have to satisfy.

\section{Recurrence relations}\label{recurr}

Let $S(1,N)=\oplus_{\eta\in \hat k^*} {S(1,N)}^{\eta}$ be the canonical decomposition of $S(1,N)$ into
isotypical components with respect to the action of $k^*=R_p^*/R_p^*(1)$. Consider the decomposition
\begin{equation}\label{s1s2}S(1,N)=(\bigoplus_{\eta=\eta^p}{S(1,N)}^\eta)\oplus(\bigoplus_{\eta\neq\eta^p}{S(1,N)}^\eta),
\end{equation}
and denote by $S_1(1,N)$ the first summand and by $S_2(1,N)$ the second one. We have that
$S_1(1,N)$ is the largest submodule of $S(1,N)$ on which the action of $D_p^*$ is abelian. Lastly,
denote the isotypical component of $S(1,N)$ with respect to the trivial character of $k^*$ by $S_0(1,N)$.
For $i\in\{0;1;2\}$ let $u_i(p,N)$ denote the dimension of $S_i(1,N)$.

\begin{prop}\label{recrel} For a prime $p$ and an integer $N\geq 1$ not divisible by $p$ we have

$(t*A_0(p,\_))(N)=u_0(p,N) - 1;$

$(t*A_1(p,\_))(N)=u_1(p,N) - (p-1);$

$(t*A_2(p,\_))(N)=u_2(p,N)/2.$

\end{prop}

Here $(t*A_i)$ denotes the convolution product between the multiplicative function $t$ introduced right before
Lemma \ref{locallem} and the function $N\rightarrow A_i(p,N)$.

\begin{proof} Observe first that the subspaces $S_0(1,N)$, $S_1(1,N)$ and $S_2(1,N)$, just defined
are invariant under the action of the Hecke operators $T_\ell$, for every prime $\ell\nmid pN$.
This follows from the fact that, for any system of eigenvalues $\Phi$, the $D^*_p$--representation ${S(1,N)}^\Phi$
is a multiple $m{(\pi_\Phi)}_p$ of the local component at $p$ of the cusp from $\pi_\Phi$ attached to $\Phi$ (cf. Prop. \ref{eigcus}),
together with the analysis of the representation of $k^*$ that may arising from ${(\pi_\Phi)}_p$ by restriction to $R^*_p$.
Again from proposition \ref{eigcus}, we have that $S_i(1,N)$ consists of the direct sum of a certain number of Hecke
isotypical components ${S(1,N)}^\Phi$ of $S(1,N)$. More precisely,

${S(1,N)}^\Phi\subset S_1(1,N)$, if ${(\pi_\Phi)}_p$ is one--dimensional;

${S(1,N)}^\Phi\subset S_2(1,N)$, if ${(\pi_\Phi)}_p$ is two--dimensional;
\\furthermore, in the first case, we have that ${S(1,N)}^\Phi\subset S_0(1,N)$ precisely when the character
of $\qu_p^*$ describing ${(\pi_\Phi)}_p$ is unramified.
Another consequence of the same proposition is that

$S{(1,N)}^{\rm Eis}\subset S_1(1,N)$, and $S{(1,N)}^{\rm Eis}\cap S_0(1,N)=\ci\cdot e_{p-1}$,
\\where $e_{p-1}$ is the constant function on $D^*_A$ equal to $1$ (cf. Prop. \ref{eisen}).

Consider the decomposition
\begin{equation}\label{decomp1}S_1(1,N)={S(1,N)}^{\rm Eis}\bigoplus_{0< d |N}\left(\bigoplus_{N(\Phi)=d}{S_1(1,N)}^\Phi\right),
\end{equation}
where the inner sum ranges through all the eigensystems $\Phi$ arising from $S(1,N)$ and giving
rise to an infinite dimensional cusp form $\pi_\Phi$ whose local component at $p$ is one dimensional,
and whose prime--to--$p$ conductor is equal to the given positive divisor $d$ of $N$.

Taking dimensions on both sides of (\ref{decomp1}), and using Proposition \ref{eisen} and Lemma
\ref{locallem}, we get
$$u_1(p,N)=(p-1)+\sum_{0<d|N}t(N/d)A_1(p,d).$$
In an analogous way, since $S_0(1,N)\cap {S(1,N)}^{\rm Eis}$ is one--dimensional, we see that
$$u_0(p,N)=1+\sum_{0<d|N}t(N/d)A_0(p,d).$$
Finally, working with the module $S_2(1,N)$, we have
$$u_2(p,N)=2\sum_{0<d|N}t(N/d)A_2(p,d).$$
The proposition follows.
\end{proof}

The relations involving the functions $A_i(p,N)$ expressed by Proposition \ref{recrel} can be solved
for the $A_i(p,N)$. The inverse of $t$, with respect to the convolution product, is infact the square $(\mu*\mu)$
of the M\"obius function. A restatement of Proposition \ref{recrel} is then

\begin{prop}\label{relarico} For a prime $p$ and an integer $N\geq 1$ not divisible by $p$ we have

$A_0(p,N)=(\mu*\mu*u_0(p,\_))(N) - \mu(N);$

$A_1(p,N)=(\mu*\mu*u_1(p,\_))(N) - (p-1)\mu(N);$

$A_2(p,N)=(\mu*\mu*u_2(p,\_))(N)/2.$

\end{prop}
Let $\delta$ be the multiplicative function whose value at every integer $n>1$ is zero, and so that $\delta(1)=1$.
In the restatement of Proposition \ref{recrel} given by Proposition \ref{relarico} we used the fact that the convolution
between $\mu$ and the constant function equal to $1$ is $\delta$, the identity for the convolution product.

To complete the proof of Theorem \ref{MT}, we will find an explicit expression for $u_i(p,N)$, for every prime $p>3$,
and every $N\geq 1$ not divisible by $p$. This task will be accomplished in section \ref{isoclasskstar}, where the isomorphism class
of the $k^*$--representation given by $S(1,N)$ will be determined.


\section{The functions $u_0$, $u_1$ and $u_2$}\label{isoclasskstar}

The double coset $\Omega(1,N)$ introduced in equation (\ref{doublecoset}) parametrizes
(in a non--canonical way) supersingular elliptic curves in characteristic $p$ with a certain extra
structure. By exploiting this correspondence, we explicitly determine, for any prime $p>3$ and any integer
$N\geq 1$ not divisible by $p$, the isomorphism class of the linear representation of $k^*$ given
by the complex space $S(1,N)$. This results in providing formulas for the functions $u_i(p,N)$, thus
completing the proof of Theorem \ref{MT}.

Any supersingular elliptic curve $E$ over $\bar k$ has a canonical and functorial model $E_0$
over $k$ specified by the requirement that the Frobenius endomorphism of $E_0$ relative to $k$
is equal to multiplication by $-p$ on the curve (cf. Thm. \ref{desc}). In particular we can deduce a $k$--structure on
the space ${\mathbf{t}_0(E)}^*$ of invariant $1$--forms, dual to the tangent space $\mathbf{t}_0(E)$
of $E$ at $0$, to which we will implicitly refer when speaking of an {\it invariant form on $E$ defined
over $k$}. The details of the next theorem are given in section \ref{isoclass} (cf. Thm. \ref{corr2}):

\begin{teo}\label{corr3} There is a correspondence between $\Omega(1,N)$ and the set of
isomorphism classes of triples $(E,\omega,x)$, where $E$ is a supersingular
elliptic curve over $\bar k$, $\omega$ is a nonzero invariant $1$--form on $E$
defined over $k$, and $x\in E[N](\bar k)$ is point of order $N$. The permutation
action of $k^*$ on $\Omega(1,N)$ is that induced by the twisted action
$$\lambda.(E,\omega,x)=(E,\lambda^p\omega,x)$$
by homotheties on the second entry of the triples considered.
\end{teo}

\begin{rem}\label{untwist} The introduction of the Frobenius--twist in the permutation action of $k^*$
as given in Theorem \ref{corr3} results from the actual way we set up the correspondence. The correspondence
can be normalized in a different way, and the twist can be removed. Observe nevertheless that the Frobenius--twisted permutation
action of $k^*$ on $\Omega(1,N)$ is isomorphic to the untwisted action, induced by $\lambda.(E,\omega,x)=(E,\lambda\omega,x)$.
This follows from the fact that there is a permutation action of $D_p^*/R_p^*(1)$ on $\Omega(1,N)$ extending that of $k^*$.
\end{rem}

The following elementary lemma on (not necessarily supersingular) elliptic curves over a field of
characteristic $p$ will be used in the sequel. Here the assumption $p>3$ plays a role.

\begin{lem}\label{wek} Let $E$ be an elliptic curve over a field $\kappa$ of characteristic $p$,
and let $W_E$ be the group of $\kappa$--automorphisms of $E$. If $p>3$, then the group homomorphism
$d_0:W_E\rightarrow\kappa^*$ describing the natural action of $W_E$ on the tangent space
$\mathbf{t}_0(E)$ (resp. the cotangent space ${\mathbf{t}_0(E)}^*$) of $E$ at $0$ is injective.
\end{lem}
\begin{proof} Let $\sigma\in W_E$ be such that the differential $d_0(\sigma)$ is equal to $1$.
Then by the linearity of the differential with respect to the addition on $E$ we have
$d_0({\rm id_E}-\sigma)=0$. Therefore ${\rm id_E}-\sigma$ is either zero or an inseparable
isogeny, in both cases $p$ divides the degree of ${\rm id_E}-\sigma$. For any pair of endomorphisms
$a$, $b$ of $E$ over $\kappa$, we have the Cauchy--Schwarz inequality (cf. \cite{Si}, III Cor. 6.3, V Lem. 1.2)
$$\deg(a+b)\leq \deg(a)+\deg(b)+2\sqrt{\deg(a)\deg(b)}.$$
In particular
$\deg({\rm id_E}-\sigma)\leq 1+1+2=4$, and its divisibility by $p>3$ implies that
$\deg({\rm id_E}-\sigma)$ is zero, and $\sigma={\rm id_E}$.
\end{proof}

\begin{rem}\label{injects}
If $E$ is a supersingular elliptic curve over $\bar k$, then from the lemma and from the existence of the functorial $k$--model $E_0$,
it follows that the action of the group $W_E$ of $\bar k$--automorphisms on ${\mathbf{t}_0(E_0)}^*$ is described by an injection
$W_E\hookrightarrow k^*$.
\end{rem}

Let $E$ be an elliptic curve over $\bar k$ with automorphism group $W_E$.
Since the characteristic $p$ of $k$ is assumed to be $>3$, we have that $W_E$ is cyclic
of order $4$, $6$, or $2$ according to whether $j_E$ is $1728$, $0$, or none of the
previous values respectively (cf.\cite{Si}, III \S 10). The next lemma indicates
when the two isomorphism classes of elliptic curves with extra automorphisms
are supersingular.

\begin{lem}\label{extraaut} If $p>3$, then $1728$ is a supersingular $j$--invariant if and only if
$p\equiv 3$ mod $4$, whereas $0$ occurs as a supersingular invariant
precisely when $p\equiv 2$ mod $3$.
\end{lem}
\begin{proof} Let $E$ be an elliptic curve over $\bar k$ with $j_E=1728$, and let $i\in\bar\qu$
be a primitive fourth root of unity. There exists an injective ring homomorphism
$\z[i]\hookrightarrow\enn_{\bar k}(E)$, sending $i$ to one of the two automorphisms of $E$ of order $4$,
we can deduce an embedding $\qu(i)\subset\enn_{\bar k}(E)\otimes\qu$. If $E$ is supersingular,
then $\enn_{\bar k}(E)\otimes\qu$ is a quaternion algebra over $\qu$ that is ramified at $p$,
therefore $\qu(i)\otimes\qu_p$ is a field, equivalently $\qu(i)$ is not split at $p$. On the other
hand if $E$ is ordinary, then $E(\bar k)$ has a (unique) subgroup $C$ of order $p$ and we can
deduce a non trivial ring homomorphism $\z[i]\rightarrow\enn(C)=\f_p$. Which is to say that $p$ splits
in $\qu(i)$, since it cannot ramify for $p>3$. Since $p$ splits in $\qu(i)$ if and only if $p\equiv$ 1 mod $4$ we obtain the first
part of the lemma.

The second part is analogous, one has to replace the fourth root of unity $i$ by a primitive sixth
root of unity $\rho\in\bar\qu$.
\end{proof}

Let now $E$ be a supersingular elliptic curve over $\bar k$. If $N$ is any integer $\geq 1$ with $p\nmid N$,
define $\Sigma_E(1,N)$ to be the set of isomorphisms classes of pairs $(\omega,x)$, where $\omega$ is a
nonzero invariant $1$--form on $E$ defined over $k$, and $x$ is a $\bar k$--valued point of $E$ of order $N$.
Two such pairs $(\omega, x)$ and $(\eta, y)$ are isomorphic if there exists $u\in W_E$ such that
$u.(\omega, x):=(u^*\omega,u(x))$ is equal to $(\eta,y)$. Let moreover $S_E(1,N)$ be the set of complex
valued functions on $\Sigma_E(1,N)$.

The group $k^*$ permutes the elements of $\Sigma_E(1,N)$, the action is induced by
$$\lambda.(\omega,x)=(\lambda\omega,x).$$
The space $S_E(1,N)$ is therefore a linear representation of $k^*$ via the formula
$$\lambda.\varphi(\omega,x)=\varphi(\lambda.(\omega,x)).$$
It follows from Theorem \ref{corr3} that there is a decomposition
$$S(1,N)=\bigoplus_{E}S_E(1,N),$$
where the sum ranges through the isomorphism classes of supersingular elliptic curves $E$
over $\bar k$.

Consider now the set $\Sigma_E(N)$ of isomorphisms classes of $\bar k$--valued points of $E$
of order $N$, by which we mean the set of $W_E$--orbits given by points of $E$ of order $N$.
Let $\alpha_1,\ldots,\alpha_r\in E[N](\bar k)$ be a system of representatives of $\Sigma_E(N)$,
and let $e_1,\ldots,e_r$ be the cardinalities of the subgroups of $W_E$ given by the respective
stabilizers. Notice that each $e_i$ divides the order of $k^*$, since $W_E$ injects into
$k^*$ (cf. Rmk. \ref{injects}).

If $d$ is an integer dividing $|k^*|=p^2-1$ then set
$$I_d={\rm Ind}_{C_d}^{k^*}1,$$
where $C_d$ is the subgroup of $k^*$ of order $d$, and $1$ denotes
its trivial, one--dimensional complex representation.

\begin{lem}\label{actionE} There is an isomorphism
$$S_E(1,N)\simeq\bigoplus_{1\leq i\leq r} I_{e_i}.$$
\end{lem}
\begin{proof} Using lemma \ref{wek}, the group $W_E$ will be identified with a subgroup of $k^*$.
Consider the permutation action of $k^*$ on $\Sigma_E(1,N)$. Any $k^*$--orbit
can certainly be represented by elements of the form $(\omega,\alpha_i)$. Moreover,
if $(\omega,\alpha_i)$ and $(\eta,\alpha_j)$ are in the same orbit then $i=j$.

In order to prove the lemma it is then enough to show that the stabilizer of
any class of $\Sigma_E(1,N)$ represented by $(\omega,\alpha_i)$ is the subgroup of
$k^*$ of order $e_i$. This follows readily from the fact that for $\lambda\in k^*$ we have that
$(\lambda\omega,\alpha_i)$ defines the same element as $(\omega,\alpha_i)$ does
if and only if there is an automorphism $u\in W_E$ such that
$$u.(\lambda\omega,\alpha_i)=(\omega,\alpha_i).$$
Therefore $u\alpha_i=\alpha_i$ and $\lambda$ has to belong to the subgroup
of $k^*$ given by the subgroup of $W_E$ stabilizing $\alpha_i$.
\end{proof}

If $E$ is any elliptic curve over $\bar k$, and $N$ any integer not divisible by $p$, we compute
the integers $e_1,\ldots, e_r$. Let $\psi(N)$ be the number of elements of ${(\z/N\z)}^2$ of order $N$.

\begin{lem}\label{ei} The sequence of integers $e_1,\ldots, e_r$ is given, up to permutation, by the
following table:

\begin{tabular}{c|cccc}
&$|W_E|=2$&$|W_E|=4$&$|W_E|=6$\\
\hline
$N=1$&$e_1=2$&$e_1=4$&$e_1=6$\\
$N=2$&$e_i=2$, $r=3$&$e_1=2$, $e_2=4$&$e_1=2$\\
$N=3$&$e_i=1$, $r=4$&$e_i=1$, $r=2$&$e_1=1$, $e_2=3$\\
$N>3$&$e_i=1$, $r=\psi(N)/2$&$e_i=1$, $r=\psi(N)/4$&$e_i=1$, $r=\psi(N)/6$\\
\end{tabular}
\end{lem}
\begin{proof} If $E$ does not have extra automorphisms, then the action of
$W_E$ on $E[N](\bar k)$ is by inversion, and the first column of the table is readily established,
taking in account that $\psi(2)=3$ and that $\psi(3)=8$.

If $|W_E|=4$, and $\ell$ is a prime $\neq p$, then a generator $\sigma\in W_E$ acts on the
$\ell$--adic Tate module $T_\ell(E)$ of $E$ via a $\z_\ell$--linear, order $4$ automorphism $\sigma_\ell$.
Let $m=m_{\sigma_\ell}\in\gl_2(\z_\ell)$ be the matrix representing $\sigma_\ell$ with
respect to the choice of a $\z_\ell$--basis of $T_\ell(E)$.
Since $m$ satisfies the polynomial $x^2+1$ and it cannot act a scalar (for $\enn_{\bar k}(E)$
injects into $\enn_{\z_\ell}(T_\ell(E))$), we see that $x^2+1$ is its characteristic polynomial.
If $\ell\neq 2$, then the roots of the reduction mod $\ell$ of $x^2+1$ are distinct, therefore if $N$ is
not a power of $2$ other than $1$, the group $W_E$ acts without fixed points on the $\bar k$--valued points
of $E$ of order $N$.

Assume then that $\ell=2$, and let $m$ be given by
$$m=\begin{pmatrix}
a&b\\
c&d\\
\end{pmatrix}.$$
The mod $2$ reduction $\bar m$ of $m$ cannot be trivial, for $\ell$ dividing $1-m$ would
prevent $m$ from having order $2$, and $\bar m$ has order $2$. It follows that there is a basis
$(e_1,e_2)$ of $T_\ell(E)$ such that the matrix $m$ representing $\sigma_\ell$ has the entries
$a$ and $d$ both divisible by $2$. This implies that the pair $(e_1,\sigma_\ell\cdot e_1)$ is also
a basis of $T_\ell(E)$, thus, up to conjugation, we can assume $m$ be given by
$$m=\begin{pmatrix}
0&-1\\
1&0\\
\end{pmatrix}.$$
It now readily follows that $\sigma_\ell$ fixes a point of order $N=2^n$ if and only if $n<2$,
moreover, the action of $W_E$ on points of order $2$ is specified by the integers
$e_1=4$ and $e_2=2$. This complete the proof of the statements in the second
column of the table.
The case $|W_E|=6$ can be treated similarly and the details are omitted.
\end{proof}

\begin{rem} If $\tilde E$ is one of the two elliptic curves over $\bar\qu_p$ with extra automorphisms, then
$\tilde E$ is known to have good reduction, since $j_{\tilde E}$ is an algebraic integer (an integer, in fact). The
special fiber $\tilde E_p$ at $p$ is an elliptic curve over $\bar k$ with extra automorphisms, and the
group ${\rm Aut}(\tilde E)$ is identified with ${\rm Aut}(\tilde E_p)$
in a natural way (recall that $p>3$). For $\ell\neq p$, the Tate modules $T_\ell(\tilde E)$ and
$T_\ell(\tilde E_p)$ are isomorphic as modules over ${\rm Aut}(\tilde E)\simeq{\rm Aut}(\tilde E_p)$
and one could have proved Lemma \ref{ei} by reducing the calculation to the complex case,
where the study of a simple picture leads to the computation of the integers $e_i$ in the table.
\end{rem}

The number of supersingular $j$--invariants characteristic $p>3$ is given by a well--known formula
due to Eichler and Deuring (cf. \cite{Gr}, \S 1):

\begin{tabular}{c|cccc}
$p$ mod $12$&$1$&$5$&$7$&$11$\\
\hline
$h$&$(p-1)/12$&$(p+7)/12$&$(p+5)/12$&$(p+13)/12$\\
\end{tabular}

This formula, together with Lemmas \ref{extraaut}, \ref{actionE} and \ref{ei}, yields:
\begin{prop}\label{explicit} The isomorphism class of the linear representation $S(1,N)$
of $k^*$ is specified by the following table:

\begin{tabular}{c|cccc}
$p$ ${\rm mod}$ $12$&$1$&$5$&$7$&$11$\\
\hline
$N=1$&$\dfrac{p-1}{12} I_2$&$\dfrac{p-5}{12} I_2+I_6$&$\dfrac{p-7}{12} I_2+I_4$&$\dfrac{p-11}{12} I_2+I_6+I_4$\\
$N=2$&$\dfrac{p-1}{4} I_2$&$\dfrac{p-1}{4} I_2$&$\dfrac{p-3}{4} I_2+I_4$&$\dfrac{p-3}{4} I_2+I_4$\\
$N=3$&$\dfrac{p-1}{3} I_1$&$\dfrac{p-2}{3} I_1+I_3$&$\dfrac{p-1}{3} I_1$&$\dfrac{p-2}{3} I_1+I_3$\\
$N>3$&$\psi(N)\dfrac{p-1}{24} I_1$&$\psi(N)\dfrac{p-1}{24} I_1$&$\psi(N)\dfrac{p-1}{24} I_1$&$\psi(N)\dfrac{p-1}{24} I_1$\\
\end{tabular}

Moreover, we have:

$u_0(p,N)=\psi(N)\dfrac{p-1}{24}+\dfrac{\Delta(p,N)}{p-1}$;

$u_1(p,N)=\psi(N)\dfrac{{(p-1)}^2}{24}+\Delta(p,N)$;

$u_2(p,N)=\psi(N)\dfrac{{(p-1)}^2p}{24}-\Delta(p,N)$.
\end{prop}

The second part of the proposition follows from the table in the first part, after an elementary
(but a bit lengthy) calculation. The error term $\Delta(p,N)$ was defined in section \ref{statementteo}.

By plugging the expressions for the functions $u_i$ obtained in Proposition \ref{explicit} in
the recurrence relation given by Poposition \ref{recrel}, we see that the proof of Theorem \ref{MT}
is complete, once that the equality $(\mu*\mu*\psi)=r$ is verified.

\section{An isogeny class of supersingular elliptic curves}\label{isoclass}


Let $p$ be a prime number, and $k$ a finite field with $p^2$ elements.
%
%
The Honda--Tate theory of abelian varieties over finite fields guarantees the existence
of a $k$--isogeny class $\mathcal{C}_{-p}$ of elliptic curves over $k$ whose objects $A$
are those for which the equality
\begin{equation}\label{eq0}\pi_A=-p
\end{equation}
holds in the ring $\enn_k(A)$ (cf. \cite{T}, Th\'eor\`eme $1$). Here $\pi_A:A\rightarrow A$ is
the geometric Frobenius endomorphism of $A$ relative to $k$, i.e., $\pi_A$ is the identity on
the topological space underlying $A$, and is given by $s\rightarrow s^{|k|}$ on sections.
Furthermore, for any such $A$ the division algebra $\enn_k(A)\otimes\qu$ is ``the'' $\qu$--quaternion
ramified at $p$ and infinity (cf. Tate, loc. cit.), and the elliptic curve $A$ is {\it supersingular}. In fact
any supersingular elliptic curve over an algebraic closure $\bar k$ of $k$ admits a canonical and
functorial descent to $\mathcal{C}_{-p}$. It can be shown that (cf. \cite{BGP}, Lemma 3.21
for a brief sketch of the proof):

\begin{teo}\label{desc} The base extension functor $A\rightarrow A\otimes_k\bar k$ induces an equivalence of
categories between $\mathcal{C}_{-p}$ and the isogeny class of supersingular elliptic curves over $\bar k$.
\end{teo}

In this section we explain how isomorphism classes of objects of $\mathcal{C}_{-p}$, equipped with
some extra structure, are parametrized by a certain double coset arising from the adelic points of
the multiplicative group of the quaternion algebra above. To describe such correspondence we make
essential use of results of Tate (cf. \cite{WM}, Thm. $6$), describing the local structure at every
prime $\ell$ of the module $\ho_k(A,B)$ of homomorphisms between two abelian varieties $A$, $B$ over
a finite field $k$.

Since the correspondence is classical, this section will probably not appear in a more definitive
version of this paper. We first learnt of this correspondence in \cite{Se}, what is included here is the
outcome of our effort in understanding all the details of its proof. Ghitza in \cite{Gh} explains, and generalizes, a correspondence
very similar to the one considered here. Our method is similar to his.

The notation adopted throughout the section is as follows. The Galois group of a fixed algebraic closure
$\bar k$ of $k$ over $k$ is denoted by $G_k$. For an object $A$ of $\mathcal{C}_{-p}$, and a prime
number $\ell$, we let $T_\ell(A)$ denote the $\ell$--adic Tate module of $A$ for $\ell\neq p$, and
the contravariant Dieudonn\'e module attached to the $p$--divisible group $A[p^\infty]$ of $A$ for
$\ell=p$ (cf. \cite{W}, Ch. 1; \cite{De}, Ch. III). For any prime $\ell$, we set $V_\ell(A)=T_\ell(A)
\otimes_{\z_\ell}\qu_\ell$. If $\ell\neq p$ then $V_\ell(A)/T_\ell(A)$ is identified with the Galois
module $A[\ell^\infty]$ of $\bar k$--valued points of $A$ of $\ell$--power torsion. If $\psi:A
\rightarrow B$ is a morphism in $\mathcal{C}_{-p}$, then $\psi_\ell:T_\ell(A)\rightarrow T_\ell(B)$
denotes the corresponding morphism of Tate modules for $\ell\neq p$, and $\psi_p:T_p(B)\rightarrow T_p(A)$
that of Dieudonn\'e modules. The morphism of $\qu_\ell$--vector spaces deduced from
$\psi_\ell$ is denoted by $V_\ell(\psi)$, while $\psi[\ell^\infty]$ denotes, for $\ell\neq p$,
the morphism from $A[\ell^{\infty}]$ to $B[\ell^\infty]$ induced by $\psi$. It is a basic fact that
$\psi\neq 0$ if and only if $\psi_\ell$ is injective for some (all) $\ell$, if and only if
$V_\ell(\psi)$ is an isomorphism for some (all) $\ell$. If $\psi$ is an isogeny, then $\ker_\ell(\psi)$
denotes the finite subgroup of $A$ given by the $\ell$--primary part of $\ker(\psi)$. If $\ell\neq p$,
then $\ker_\ell(\psi)$ will be identified with the Galois module given by $\coker(\psi_\ell)$, on the
other hand $\coker(\psi_p)$ is the finite length Dieudonn\'e module associated to $\ker_p(\psi)$.

\subsection{The endomorphism ring}\label{endring} Let $E$ be any object of $\mathcal{C}_{-p}$, denote the ring $\enn_k(E)$ by $R$,
and the $\qu$--algebra $\enn_k(E)\otimes\qu$ by $D$. If $\ell$ is any prime number, set $R_\ell=R\otimes\z_\ell$
and $D_\ell=D\otimes\qu_\ell$. As mentioned above, $D$ is a central, division algebra over $\qu$ such that $D_\ell$
is isomorphic to the matrix algebra $M_2(\qu_\ell)$ when $\ell\neq p$, and to ``the'' central, division algebra
over $\qu_\ell$ of rank four when $\ell = p$. In this section we show that
$R$ is a {\it maximal} order of $D$ or, equivalently, that $R_\ell$ is a maximal $\z_\ell$--order in $D_\ell$, for
all $\ell$.

If $\ell$ is a prime $\neq p$, then $T_\ell(E)$ is a free $\z_\ell$--module of rank two on which $G_k$ acts continuously
and in a natural way. The action of the arithmetic Frobenius ${\rm Frob}_k\in G_k$ on $T_\ell(E)$ is the {\it same} as that
given by $\pi_{E,\ell}$, i.e., that induced by $\pi_E$ via functoriality of the Tate module. Therefore the Galois module
structure of $T_\ell(E)$ is specified by the requirement that ${\rm Frob}_k$ act as multiplication by $-p$. In particular we see that
\begin{equation}\label{eq1}\enn_{\z_\ell[G_k]}(T_\ell(E))=\enn_{\z_\ell}(T_\ell(E)).
\end{equation}
The natural map $\iota_\ell:R\rightarrow\enn_{\z_\ell[G_k]}(T_\ell(E))$ sending $r\in R$ into the induced morphism
$r_\ell$ of Tate modules extends by continuity to a map on $R_\ell$, also denote by $\iota_\ell$. A special case of
a theorem of Tate (cf. \cite{T2} or \cite{WM}) yields:

\begin{teo}\label{tateendell} The map $\iota_\ell:R_\ell\rightarrow\enn_{\z_\ell}(T_\ell(E))$ is an isomorphism, and
$R_\ell$ is a maximal order of $D_\ell$.
\end{teo}
%
%
In order to study $R_p$ it is customary to work with the
contravariant Dieudonn\'e module $T_p(E)$ attached to the $p$--divisible group $E[p^\infty]$ of $E$.
If $W=W(k)$ is the ring of Witt vectors of $k$, recall that the Dieudonn\'e ring $\mathcal{A}=\mathcal{A}_k$ over
$k$ is the non--commutative, polynomial ring in two variables $W[F,V]$ subject to the relations

$FV=VF=p$;

$F\lambda=\lambda^\sigma F$, $V\lambda^\sigma=\lambda V$;\\
where $\lambda\in W$ is any element, and $\lambda\rightarrow\lambda^\sigma$ is the automorphism of $W$ inducing the absolute Frobenius on the
residue field $k$. Notice that in this special case where $|k|=p^2$ we have that $F$ and $V$ have the same
semi--linear behavior with respect to the action on $W$, since $\sigma=\sigma^{-1}$.

The Dieudonn\'e module $T_p(E)$ is a left $\mathcal{A}$--module that is finite and free of rank two as $W$--module.
Equation (\ref{eq0}) implies that $F^2$ acts as multiplication by $-p$, therefore the $k$--semi--linear
endomorphism of $T_p(E)/pT_p(E)$ induced by $F$ is nonzero and nilpotent. Using this one shows
that there exists a $W$--basis $(e_1,e_2)$ of $T_p(E)$ such that

$F(e_1)=-pe_2$;

$F(e_2)=e_1$;
\\moreover we must have that $V=-F$ (cf. \cite{Gh2} for more details). The nonzero $\mathcal{A}$--submodules of $T_p(E)$ that are $W$--lattices of $V_p(E)$
are those given by $F^nT_p(E)$, for $n\geq 0$. In terms of the arithmetic of $E$, this implies that every
finite, closed subgroup of $E$ of $p$--power rank is the kernel of a chain of successive, alternating 
applications of the absolute Frobenius morphisms $F_E:E\rightarrow E^{(p)}$ and
$F_{E^{(p)}}:E^{(p)}\rightarrow E^{(p^2)}=E$. Notice that equation (\ref{eq0}) says that $F_{E^{(p)}}F_{E}=-p$.

The ring $R$ acts on the right of $T_p(E)$ by functoriality, and determines a ring homomorphism
$\iota_p:R^{\rm op}\rightarrow\enn_{\mathcal{A}}(T_p(E))$ which extends by continuity to $R_p^{\rm op}$.
There is a version at $p$ of Theorem (\ref{tateendell}), also due to Tate (cf. \cite{WM}, Thm. $6$):
\begin{teo}\label{tatep} The map $\iota_p:R_p^{\rm op}\rightarrow\enn_{\mathcal{A}}(T_p(E))$ is an isomorphism.
\end{teo}
The ring $R_p^{\rm op}$ is thus identified with the ring of $W$--linear endomorphisms of $T_p(E)$ commuting
with $F$. Using the coordinate system induced by the previous basis $(e_1,e_2)$ of $T_p(E)$, we readily compute that
$R_p^{\rm op}$ is described by matrices of the form
$$\begin{pmatrix}
a^{\sigma}& - b\\
pb^\sigma&a\\
\end{pmatrix},$$
where $a,b\in W$. Letting $a$ and $b$ vary in the degree two, unramified extension $L=W[1/p]$ of $\qu_p$, leads to a matrix
description of $D_p^{\rm op}=R_p^{\rm op}\otimes_{\z_p}\qu_p$ whose trace ($\trace$) and determinant ($\det$)
respectively give the reduced trace and reduced norm of $D_p^{\rm op}$ (cf. \cite{Vi}, I.1). The involution
$x\rightarrow\trace(x)-x$ gives an isomorphism $D_p^{\rm op}\xrightarrow{\sim} D_p$, sending $R_p^{\rm op}$ to 
$R_p$. Its effect on the matrix description is
$$\begin{pmatrix}
a^{\sigma}&-b\\
pb^\sigma&a\\
\end{pmatrix}\longrightarrow\begin{pmatrix}
a&b\\
-pb^\sigma&a^\sigma\\
\end{pmatrix}.$$
%
The ring $R_p$ is the unique {\it maximal} order of $D_p$ as it consists of all the elements with both reduced trace and
reduced norm valued in $\z_p$. If $\nu_p:\qu_p\rightarrow\z\cup\{\infty\}$ is the valuation of $\qu_p$, normalized in such a way that
$\nu_p(p)=1$, then the function $R\ni r\rightarrow\nu_p(\det(r))\in\z\cup\{\infty\}$ has the formal properties of a
discrete valuation (\cite{Vi}, II.1). Therefore $R$ has a unique, two--sided, maximal, principal ideal $\mathfrak{m}$ given by all the
elements whose determinant belongs to $(p)\subset\z_p$, and generated by element whose determinant generates $(p)$ in $\z_p$.
The residue field $R/\mathfrak{m}$ has order $p^2$.

\subsection{Isogenies in $\mathcal{C}_{-p}$ and ideals of $\enn_k(E)$}\label{subiso}
In this section we prove two theorems useful to describe a correspondece between
isogenies of $\mathcal{C}_{-p}$ whose source is a fixed object $E$, and nonzero, left ideals
of $R=\enn_k(E)$. They both follow from the results of Tate that we already encountered.

%

Let $\psi:E\rightarrow E_\psi$ be any isogeny in $\mathcal{C}_{-p}$, consider the $R$--left ideal
$$I_\psi=\{r\in R: r=r'\psi,\hphantom{.}{\rm for}\hphantom{.}{\rm some}\hphantom{.}r'\in \ho_k(E_\psi,E)\}$$
consisting of all the endomorphisms of $E$ trivial on the finite subgroup $\ker(\psi)$.
Pull--back by $\psi$ gives an isomorphism $\psi^*:\ho(E_\psi,E)\xrightarrow{\sim}I_\psi$ of left
$R$--modules. Notice that $I_\psi$ is a $\z$--lattice of $D$, since for example $\deg(\psi)\in I_\psi$.
If $\ell$ is any prime, the module $I_\psi\otimes\z_\ell$ will be denoted by $I_{\psi,\ell}$,
it is a left ideal of $R_\ell$ that is also a $\z_\ell$--lattice of $D_\ell$. The isomorphism $\psi^*$ above
induces an isomorphism $\ho(E_\psi,E)\otimes\z_\ell\xrightarrow{\sim} I_{\psi,\ell}$, tacitly used in what follows.

\begin{rem}\label{locpri} From the maximality of $R_\ell$, it follows that $I_{\psi,\ell}$ is {\it principal}
for any prime $\ell$ (cf. \cite{Vi} II.1, II.2). If $\ell\neq p$, there exists a $\z_\ell$--lattice
$\Lambda_0$ of $V_\ell(E)$ {\it containing} $T_\ell(E)$ such that the isomorphism $\iota_\ell$
(cf. Thm. \ref{tateendell}) sends $I_{\psi,\ell}$ to the left ideal $\enn_{\z_\ell}(\Lambda_0,T_\ell(E))$.
In fact we will see that $\Lambda_0$ is the pull--back of $T_\ell(E_\psi)$ via $V_\ell(\psi)$ (cf. proof of Thm. \ref{teo1}).
For $\ell=p$, there exists a $W$--lattice $M_0$ in $V_p(E)$ {\it contained} in
$T_p(E)$ such that the isomorphism $\iota_p$ (cf. Thm. \ref{tatep}) sends $I_{\psi,p}$ to the right ideal
$\enn_\mathcal{A}(T_p(E),M_0)$. From the proof of Theorem \ref{teo1} it will follow that
$M_0$ is the isomorphic image of $T_p(E_\psi)$ with respect to $V_p(\psi):V_p(E_\psi)\rightarrow V_p(E)$.
For any $\ell$, the left ideal $I_{\psi,\ell}$ is generated by any of its elements of reduced norm with
minimal valuation.
\end{rem}

Before showing that the isogeny $\psi$ can essentially be recovered from $I_\psi$ (cf. Thm. \ref{teo1}),
we make a few observations. Let $\ell$ be a prime $\neq p$, from the isogeny $\psi:E\rightarrow E_\psi$ we can deduce a
commutative diagram of $G_k$--modules with exact rows
$$\xymatrix{ 0\ar[r] &T_\ell(E)\ar[r]\ar[d]^{\psi_\ell}& V_\ell(E)\ar[r]\ar[d]^{V_\ell(\psi)}
&E[\ell^\infty]\ar[r]\ar[d]^{\psi[\ell^\infty]}&0\\
0\ar[r] &T_\ell(E_\psi)\ar[r]& V_\ell(E_\psi)\ar[r]&E_\psi[\ell^\infty]\ar[r]&0.}$$
The pull--back of the Tate module $T_\ell(E_\psi)$ via the isomorphism $V_\ell(\psi)$
identifies the former with a $\z_\ell$--lattice of $V_\ell(E)$ containing $T_\ell(E)$,
and that will be denoted by $\Lambda_{\psi_\ell}$.
Under this identification, the map $\psi_\ell$ clearly corresponds to the inclusion
$T_\ell(E)\subset\Lambda_{\psi_\ell}$.
Since the two rows of the diagram are exact, $V_\ell(\psi)$ induces an identification
$\coker(\psi_\ell)\simeq\ker(\psi[\ell^\infty])$, or, equivalently,
\begin{equation}\label{cokerell}
\Lambda_{\psi_\ell}/T_\ell(E)\simeq\ker_\ell(\psi).
\end{equation}
The $\z_\ell$--lattice $\Lambda_\psi$ depends on $\psi$ in a functorial way, in a sense that can
be made precise by working with the category $\mathcal{C}_E$ that will be introduced in the next section.
We observe that if $\varphi:E\rightarrow E_\varphi$ is another isogeny, and $u:E_\psi\rightarrow E_\varphi$ is a morphism, then
the map ${V_\ell(\varphi)}^{-1}V_\ell(u)V_\ell(\psi):V_\ell(E)\rightarrow V_\ell(E)$ sends $\Lambda_\psi$ to $\Lambda_\varphi$
and there is a commutative diagram
$$\xymatrix{ \Lambda_\psi\ar[rr]^{{V_\ell(\varphi)}^{-1}V_\ell(u)V_\ell(\psi)}\ar[d]^{\simeq}&&\Lambda_\varphi\ar[d]^{\simeq}\\
T_\ell(E_\psi)\ar[rr]^{u_\ell}&&T_\ell(E_\varphi)}$$
where the vertical morphisms are the natural identifications. In particular, if
$u$ is an isogeny, the cokernel of the top horizontal map of the diagram is isomorphic to $\ker_\ell(u)$ in a natural way.

If $\ell=p$, the situation is formally analogous, after replacing the covariant Tate module by the contravariant
Dieudonn\'e module. The morphism $\psi_p:T_p(E_\psi)\rightarrow T_p(E)$ identifies $T_p(E_\psi)$
with an $\mathcal{A}$--submodule $M_{\psi_p}$ of $T_p(E)$ that is a $W$--lattice of $V_p(E)$. Moreover
$T_p(E)/M_{\psi_p}$ is the Dieudonn\'e module associated to the finite group $\ker_p(\psi)$. As before,
if $\varphi:E\rightarrow E_\varphi$ is an isogeny and $u:E_\psi\rightarrow E_\varphi$ is any morphism,
then the map $V_p(\psi)V_p(u){V_p(\varphi)}^{-1}:V_p(E)\rightarrow V_p(E)$ sends $M_{\varphi_p}$ to
$M_{\psi_p}$, there is a commutative diagram
$$\xymatrix{ M_{\psi_p}\ar[d]^{\simeq}&&M_{\varphi_p}\ar[ll]_{V_p(\psi)V_p(u){V_p(\varphi)}^{-1}}\ar[d]^{\simeq}\\
T_p(E_\psi)&&T_p(E_\varphi)\ar[ll]_{u_p}}$$
where the vertical morphisms are the natural identifications. Moreover the cokernel of
the top horizontal map of the diagram is the Dieudonn\'e module of $\ker_p(u)$.




\begin{teo}\label{teo1} For any isogeny $\psi:E\rightarrow E_\psi$ and any prime number $\ell$, there exists
an isogeny $r':E_\psi\rightarrow E$ whose degree is prime to $\ell$. Equivalently, there exists an isogeny
$r\in I_\psi$ so that $\ker_\ell(r)=\ker_\ell(\psi)$. In particular, $\ker(\psi)$ coincides with the largest $k$--subgroup of
$E$ killed by every element of $I_\psi$.
\end{teo}
\begin{proof} We begin by proving the theorem in the case $\ell\neq p$. Consider the commutative diagram
$$\xymatrix{ \ho(E_\psi,E)\otimes\z_\ell\ar[r]\ar[d]^{\simeq}&\ho_{\z_\ell}(T_\ell(E_\psi),T_\ell(E))\ar[d]^{\simeq}\\
I_{\psi,\ell}\ar[r]&\ho_{\z_\ell}(\Lambda_{\psi_\ell},T_\ell(E))}$$
where the horizontal maps are induced by functoriality of the Tate module, and
the vertical isomorphism on the right comes from the identification $T_\ell(E_\psi)\simeq\Lambda_{\psi_\ell}$
and is given by $u_\ell\rightarrow u_\ell V_\ell(\psi)_{|\Lambda_{\psi_\ell}}^{|T_\ell(E_\psi)}$.

Since $\ho_{\z_\ell[G_k]}(T_\ell(E_\psi),T_\ell(E))=\ho_{\z_\ell}(T_\ell(E_\psi),T_\ell(E))$,
the main theorem of \cite{T2} says that the top horizontal map of the previous
diagram is an isomorphism, thus so is the bottom one and the isomorphism
$\iota_\ell$ from Theorem \ref{tateendell} induces an identification of left ideals
\begin{equation}\label{locprincell} I_{\psi,\ell}\simeq\ho_{\z_\ell}(\Lambda_{\psi_\ell},T_\ell(E)).
\end{equation}

In particular, $I_{\psi,\ell}=R_\ell g_\ell$, where $g_\ell\in D_\ell$ is any element of
$R_\ell$ mapping $\Lambda_{\psi_\ell}$ isomorphically to $T_\ell(E)$. If we pick $r\in I_\psi$ such that its image
$r_\ell\in I_{\psi,\ell}$ is a generator then $\ker_\ell(r)=\ker_\ell(\psi)$ by (\ref{cokerell}), since the lattices
$\Lambda_{r_\ell}$ and $\Lambda_{\psi_\ell}$ in $V_\ell(E)$ coincide.

For $\ell=p$ consider, in an analogous way, the commutative diagram
$$\xymatrix{ \ho(E_\psi,E)\otimes\z_p\ar[r]\ar[d]^{\simeq}&\ho_{\mathcal{A}}(T_p(E),T_p(E_\psi))\ar[d]^{\simeq}\\
I_{\psi,p}\ar[r]&\ho_\mathcal{A}(T_p(E),M_{\psi_p})}$$
where the horizontal maps are induced by functoriality of $T_p$, and the vertical
isomorphism on the right is given by $u_p\rightarrow\psi_p u_p$. Notice that
the lower horizontal map is that induced by the anti--isomorphism
$\iota_p:R_p\rightarrow\enn_\mathcal{A}(T_p(E))$, where $\ho_\mathcal{A}(T_p(E),M_{\psi_p})$ is
regarded in a natural way as a right ideal of $\enn_\mathcal{A}(T_p(E))$.

Since both horizontal maps are isomorphisms (cf. \cite{WM}, Thm. 6), the principal
left ideal $I_{\psi,p}$ of $R_p$ is generated by any element $r_p$ such that
$\iota_p(r_p)$ defines an endomorphism of $T_p(E)$ sending $T_p(E)$ isomorphically to
$M_{\psi_p}$. Choosing now $r\in I_\psi$ such that $r_p\in I_{\psi,p}$ is a generator,
we see that $M_{r_p}=M_{\psi_p}$, therefore $\ker_p(r)=\ker_p(\psi)$ and the theorem
follows.
\end{proof}

Let now $I$ be a nonzero left ideal of $R$, consider the isogeny
\begin{equation}\label{idealisogeny}\psi_I:E\longrightarrow E/H(I),
\end{equation}
where
$$H(I)=\bigcap_{r\in I}\ker(r)$$
is the largest $k$--subgroup of $E$ that is killed by every element of $I$, and
$\psi_I$ is the canonical isogeny. The ideal $I$ can be recovered
from $\psi_I$:

\begin{teo}\label{teo2} For any nonzero left ideal $I$ of $R$ we have $I=I_{\psi_I}$.
\end{teo}
\begin{proof} Let $J=I_{\psi_I}$ be the left ideal of $R$ given by the endomorphisms of $E$ trivial on $H(I)$,
certainly $I\subset J$. For any prime $\ell$, denote by $I_\ell$ and $J_\ell$ the left ideals of $R_\ell$ obtained
from $I$ and $J$ after tensoring with $\z_\ell$. Since $I_\ell\subset J_\ell$, it will be enough to show that
$J_\ell\subset I_\ell$ for any $\ell$.

For $\ell\neq p$, as explained at the beginning of the section, the Tate module $T_\ell(E/H(I))$ is identified with
the $\z_\ell$--lattice $\Lambda_{\psi_{I,\ell}}$ of $V_\ell(E)$ containing $T_\ell(E)$ and such that the equality
$$\Lambda_{\psi_{I,\ell}}/T_\ell(E)=\ker_\ell(\psi_I),$$
holds in the $\ell$--divisible group $E[\ell^\infty]$. Moreover, the isomorphism $\iota_\ell:R_\ell\xrightarrow{\sim}
\enn_{\z_\ell}(T_\ell(E))$ induces an identification $J_\ell\simeq\ho_{\z_\ell}(\Lambda_{\psi_{I,\ell}},T_\ell(E))$,
as was shown in the proof of Theorem \ref{teo1}.

On the other hand, $I_\ell$ is a nonzero left ideal of $R_\ell=\enn_{\z_\ell}(T_\ell(E))$ and
there exists a $\z_\ell$--lattice $\Lambda_0$ of $V_\ell(E)$, containing $T_\ell(E)$, such that
$I_\ell=\ho_{\z_\ell}(\Lambda_0,T_\ell(E))$; since $I_\ell\subset J_\ell$, we have
$\Lambda_0\supset\Lambda_{\psi_{I,\ell}}$.

Now, by definition of $\psi_I$ we have
$$\ker_\ell(\psi_I)=\bigcap_{0\neq r\in I}\ker_\ell(r),$$
or, reformulating,
\begin{equation}\label{lam}\Lambda_{\psi_{I,\ell}}=\bigcap_{0\neq r\in I}{V_\ell(r)}^{-1}(T_\ell(E)).
\end{equation}
The right hand side of (\ref{lam}) certainly contains $\Lambda_0$, therefore
$\Lambda_0\subset\Lambda_{{(\psi_I)}_\ell}$ and $I_\ell\supset J_\ell$.

For $\ell=p$ the proof is similar and is omitted. As in the case $\ell\neq p$, the key fact is
that $I_p$ is known {\it a priori} to be principal, thanks to the maximality of $R_p$ (cf. Remark \ref{locpri}).
\end{proof}

\begin{rem} Theorem \ref{teo2} is a special case of a more general result of Waterhouse (cf. \cite{W}, Thm. 3.15).
\end{rem}

\subsection{An equivalence of categories}\label{eqcat} We interpret the result of the previous section in a formal way,
by showing the existence of an anti--equivalence of categories (cf. Thm. \ref{functor}).

Let $E$ be a fixed object of $\mathcal{C}_{-p}$, and let $\mathcal{C}_{E}$ be the category of {\it isogenies} of
$\mathcal{C}_{-p}$ with source $E$, that is the category whose objects are isogenies in $\mathcal{C}_{-p}$ of
the form $\psi:E\rightarrow E_\psi$, and whose sets of morphisms $\ho(\psi,\varphi)$ between two objects is
simply $\ho(E_\psi,E_\varphi)$.

Let $\psi:E\rightarrow E_\psi$ be any object of $\mathcal{C}_{E}$, the natural identification
$I_\psi=\ho(E_\psi,E)=\ho(\psi,{\rm id}_E)$ makes clear that $I_\psi$ depends functorially
on $\psi$; more precisely the assignment $\psi\rightarrow I_\psi$ defines a contravariant functor $\mathcal{I}$ from
$\mathcal{C}_{-p}$ to the category $\mathcal{P}_R^{(1)}$ of projective left $R$--modules of rank one.
We clarify that if $u\in\ho(\psi,\varphi)$ is a morphism in $\mathcal{C}_E$, then $\mathcal{I}(u):I_\varphi\rightarrow I_\psi$
is constructed as follows. If $r=r''\varphi \in I_\varphi$, consider the commutative diagram

$$\xymatrix{E\ar[dd]^{\psi}&&E\ar[rr]^{r}\ar[dd]^{\varphi}&&E\\
&&&&\\
E_\psi\ar[rr]^{u}&&E_\varphi\ar[rruu]^{r''}&&\\}$$
the value of $\mathcal{I}(u)$ on $r$ is the composition $r''u\psi\in I_\psi$.

From the results of the previous section we can draw the following consequence:

\begin{teo}\label{functor} The functor $\mathcal{I}$ sending $\psi$ to $I_\psi$ induces an anti--equivalence of categories
between $\mathcal{C}_{E}$ and $\mathcal{P}_R^{(1)}$.
\end{teo}
\begin{proof} We show that any object of $\mathcal{P}_R^{(1)}$ is isomorphic to one of the form $\mathcal{I}(\psi)$, and
that $\mathcal{I}$ is fullyfaithful.

Any nonzero projective left $R$--module $P$ of rank one is isomorphic to a finitely generated, left $R$--submodule of $D$,
therefore to a nonzero left ideal $I$ of $R$, after multiplication by a suitable integer $N\geq 1$. From Theorem \ref{teo2} we see that
every nonzero left ideal $I$ of $R$ is of the form $\mathcal{I}(\psi)$, for some $\psi$ in $\mathcal{C}_{E}$.

To see that $\mathcal{I}$ is fullyfaithful, let $\psi$ and $\varphi$ be two objects of $\mathcal{C}_{E}$, and $f:I_\varphi\rightarrow I_\psi$
any morphism of left $R$--modules. We have to show that there exists $u\in\ho(\psi,\varphi)$ such that $f=\mathcal{I}(u)$.
Since $I_\varphi$ and $I_\psi$ are lattices of $D$, $f$ extends uniquely to an endomorphism of $D$, as left $D$--module,
thus there is $\lambda\in D$ such that $f(x)=x\lambda$, for all $x\in I_\varphi$. Let $N$ be an integer $\geq 1$ such that
$N\lambda\in R$. Right multiplication by $N\lambda$ on $D$ clearly induces the morphism $Nf:I_\varphi\rightarrow I_\psi$
and we will first show that there exists $u'\in\ho(E_\psi,E_\varphi)$ such that $\mathcal{I}(u')=Nf$.

For $r=r''\varphi\in I_\varphi$ consider the commutative diagram

$$\xymatrix{E\ar[rr]^{N\lambda}\ar[dd]^{\psi}&&E\ar[rr]^{r}\ar[dd]^{\varphi}&&E\\
&&&&\\
E_\psi\ar@{-->}[rrrruu]_<<<<<<<<<<<{r'}&&E_\varphi\ar[rruu]_<<<<<<{r''}&&\\}$$
The existence of the dotted arrow $r'$ follows from the fact that $rN\lambda$ belongs to $I_\psi$ (in fact even to $NI_\psi$), and
hence factors as $r'\psi$, for some $r'\in\ho(E_\psi,E)$. Showing that there exists $u'\in\ho(E_\psi,E_\varphi)$ with $\mathcal{I}(u')=Nf$
amounts to proving that $\varphi N\lambda$ factors as $u'\psi$, for a certain $u'\in\ho(E_\psi,E_\varphi)$.

If $N\lambda=0$ this is clear, therefore we may assume that $\varphi N\lambda$ is an isogeny. By Theorem \ref{teo1} we have that
$$\ker(\varphi N\lambda)=\bigcap_{r''\in\ho(E_\varphi,E)}\ker(r''\varphi N\lambda),$$
but $r''\varphi N\lambda=r'\psi$, therefore
$$\ker(\psi)\subset\bigcap_{r''\in\ho(E_\varphi,E)}\ker(r''\varphi N\lambda).$$
The two equations readily imply $\ker(\psi)\subset\ker(\varphi N\lambda)$, which gives
the desired factorization of $\varphi N\lambda$ as $u'\psi$.

A further application of Theorem \ref{teo1} gives that $u'$ is divisible by $N$ in $\ho(E_\psi,E_\varphi)$. In fact
since $r''u'\psi\in NI_\psi$, for $r''\in\ho (E_\varphi,E)$, we have that $r''u':E_\psi\rightarrow E$ is trivial on the subgroup
$E_\psi[N]$. Since this happens for {\it all} $r''\in\ho (E_\varphi,E)$ we see that, by Theorem \ref{teo1}, $u'$ is itself
trivial on $E_\psi[N]$, and $u'=Nu$ for a unique $u\in\ho(E_\psi,E_\varphi)$. Therefore $\mathcal{I}(u)=N^{-1}\mathcal{I}(u')=f$,
this completes the proof of the theorem.
\end{proof}

\subsection{Adelic description of $\mathcal{C}_{-p}$}\label{adede} The functor $\mathcal{I}$ introduced in the previous section
admits an adelic description, thanks to the fact that $I_\psi$ is locally principal, for any $\psi$ in $\mathcal{C}_E$ (cf.
Remark \ref{locpri}). Let $\hat R$ denote $R\otimes\hat\z$ and, similarly, $\hat D$ denote $D\otimes\hat\z$, we have $\hat
R=\prod R_\ell$ and $\hat D=\prod D_\ell$, where the product is taken over all the primes $\ell$.

If $\psi$ is any object of $\mathcal{C}_E$, and $\ell$ is any prime, then $I_{\psi,\ell}=R_\ell g_\ell$, for some $g_\ell\in R_\ell$;
observe that $g_\ell\in D_\ell^*$, since $I_{\psi,\ell}$ is a $\z_\ell$--lattice of $D_\ell$. The image of the adelic element
${(g_\ell)}_\ell\in\hat R\cap\hat D^*$ in the coset space $\hat R^*\backslash\hat R\cap\hat D^*$
depends only on $\psi$, and not on the choices of the $g_\ell$, and is denoted by $a_\psi$.

Vicerversa, if $a\in\hat R^*\backslash\hat R\cap\hat D^*$ is represented by ${(g_\ell)}_{\ell}\in \hat R\cap\hat D^*$,
then $R_\ell g_\ell=R_\ell$ for almost all $\ell$, and there is a unique left ideal $I_a$ of $R$ such that $I_{a,\ell}=I_a\otimes\z_\ell$
is equal to $R_\ell g_\ell$, for all $\ell$. The isogeny corresponding to $I_a$ will simply be denoted by $\psi_a:E\rightarrow E_{\psi_a}$,
it defines an element of $\mathcal{C}_E$.

If $\Sigma$ is the set of isomorphism classes of objects of $\mathcal{C}_{E}$, or of $\mathcal{C}_{-p}$, a
consequence of Theorem \ref{functor} is:

\begin{teo}\label{tau} The assignment $\psi\rightarrow a_\psi$ induces a correspondence
$$\tau:\Sigma\xrightarrow{\sim}\hat R^*\backslash \hat D^*/D^*.$$
\end{teo}

The surjectivity of $\tau$ follows from the fact that the inclusion $\hat R\cap\hat D^*\subset \hat D^*$ induces a
bijection $\hat R^*\backslash \hat R\cap\hat D^*/D^*=\hat R^*\backslash\hat D^*/D^*$.

Let now $N\geq 1$ be a fixed integer not divisible by $p$. If $\psi:E\rightarrow E_\psi$ is an object of $\mathcal{C}_E$,
the tangent space $\mathbf{t}_0(E_\psi)$ of $E_\psi$ at the origin is a one--dimensional $k$--vector space, and
its dual $\mathbf{t}_0(E_\psi)^*$ is identified with the space of invariant $1$--form on $E_\psi$ (defined over $k$).
Consider the set $\Sigma(1,N)$ of isomorphism classes of triples $(\psi,\omega,x)$ given by an object
$\psi:E\rightarrow E_\psi$ of $\mathcal{C}_E$, a nonzero invariant $1$--form $\omega\in\mathbf{t}_0(E_\psi)^*$, and a
point $x\in E_\psi[N](\bar k)$ of order $N$. The notion of isomorphism between two such triples is clear:
$(\psi,\omega,x)$ is isomorphic to $(\varphi,\eta,y)$ if and only if there exists an isomorphism $u:E_\psi\xrightarrow{\sim} E_{\varphi}$
such that $u^*(\eta)=\omega$ and $u(x)=y$. Notice that $k^*$ acts on $\Sigma(1,N)$ via homotheties on the second
entry of the triples.

Choose now a nonzero $1$--form $\omega_0\in\mathbf{t}_0(E)^*$, and an element $x_0\in E[N](\bar k)$ of order $N$.
The ring $\hat R=\prod R_\ell$ acts on the left of the product
$$\mathbf{t}_0(E)^*\times\prod_{\ell\neq p}E[\ell^\infty]$$
by considering the natural action of $R_\ell$ on $E[\ell^\infty]$ for $\ell\neq p$, and on $\mathbf{t}_0(E)^*$ for $\ell=p$.
Notice that the action of $R_p$ on $\mathbf{t}_0(E)^*$ is given by a canonical ring homomorphism $R_p\rightarrow k$,
necessarily surjective because of the structure of $R_p$, that will be used to identify the residue field $R_p/\mathfrak{m}$
with $k$. The kernel of the corresponding character $R_p^*\rightarrow k^*$ is given by $R_p^*(1)$, the subgroup of units
congruent to $1$ modulo the maximal, two sided ideal $\mathfrak{m}$.

Let $K(1,N)$ be the open subgroup of $\hat R^*$ which stabilizes the pair $(\omega_0,x_0)$. We have a decomposition
$$K(1,N)=R_p^*(1)\times\prod_{\ell\neq p} K_\ell(N),$$
for a certain collection $K_\ell(N)$ of open subgroups of $R_\ell^*$, with $\ell\neq p$. The subgroup $R_p^*$ of $\hat R^*$
normalizes $K(1,N)$ and left translation induces an action of $R_p^*/R_p^*(1)=k^*$ on the coset space $K(1,N)\backslash\hat D^*/D^*$.

\begin{rem}\label{defineun} There is a $\z_\ell$--basis $(e_1,e_2)$ of $T_\ell(E)$, for $\ell\neq p$, such that the corresponding identification
$R_\ell^*\simeq\gl_2(\z_\ell)$ sends $K_\ell(N)$ to the subgroup
$${U_\ell(N)}^\gamma=\left\{x\in\textrm{GL}_2(\z_\ell)|x\equiv
\begin{pmatrix}
1&*\\
0&*\\
\end{pmatrix}\textrm{mod}\hphantom{x}\textrm{N}\right\},$$
which is the image of $U_\ell(N)$ (cf. section \ref{autom}) with respect to the canonical involution $\gamma$.
If $\ell^e$ is the exact power of $\ell$ dividing $N$, the basis above should be chosen in such a way
that the element of $T_\ell(E)/(N)=E[\ell^e](\bar k)$ defined by $e_1$ (mod $N$) is equal to the
$\ell$--primary component of $x_0$.
\end{rem}

\begin{teo}\label{corr2} There is a natural bijection
$$\tau_N:\Sigma(1,N)\xrightarrow{\sim}K(1,N)\backslash\hat D^*/D^*.$$
The permutation action of $k^*$ on $\Sigma(1,N)$ corresponds via $\tau_N$ to the action
induced by left translation of $R_p^*$ on $K(1,N)\backslash\hat D^*/D^*$.
\end{teo}

\begin{proof} Let $a\in K(1,N)\backslash\hat R\cap\hat D^*$ be any element, we begin by showing
how to construct a triple $(\psi_a,\omega_a,x_a)$ of the type considered out of $a$. Let $I_a$
be the left ideal of $R$ defined by the adelic coset $\bar a\in\hat R^*\backslash\hat R\cap\hat D^*$
deduced from $a$. The isogeny $\psi_a:E\rightarrow E_{\psi_a}$ is that obtained from the $I_a$
using the construction (\ref{idealisogeny}) from section \ref{subiso}.

Let $g={(g_\ell)}_{\ell}\in\hat R\cap\hat D^*$ be any element representing $a$, i.e., such that $a=K(1,N)g$.
Recall that for $\ell\neq p$ the module $T_\ell(E_{\psi_a})$ is identified with a $\z_\ell$--lattice
$\Lambda_{\psi_{a,\ell}}$ of $V_\ell(E)$ containing $T_\ell(E)$, moreover $g_\ell\in D_\ell^*$ is a
generator of the principal, $R_\ell$--left ideal $\ho_{\z_\ell}(\Lambda_{\psi_{a,\ell}},T_\ell(E))$,
and defines an isomorphism of $V_\ell(E)$ that maps $\Lambda_{\psi_{a,\ell}}$ isomorphically to
$T_\ell(E)$. Therefore we deduce, for any $n\geq 0$, an identification
\begin{equation}\label{tatemodn}E_\psi[\ell^n](\bar k)=\Lambda_{\psi_{a,\ell}}/(\ell^n)\simeq T_\ell(E)/(\ell^n)=E[\ell^n](\bar k)
\end{equation}
depending on the choice of the generator $g_\ell$ of $I_{\psi,\ell}$.


If now $\ell^e$ is the exact power of $\ell$ dividing $N$, then the $\ell$--th primary component
$x_{0,\ell}$ of $x_0$ is a point of order $\ell^e$ in $T_\ell(E)/(\ell^e)$. The inverse image
$x_{a,\ell}\in\Lambda_{\psi_{a,\ell}}/(\ell^e)$ of $x_{0,\ell}$ by the identification (\ref{tatemodn})
for $n=e$ is a point $E_{\psi_a}[\ell^e](\bar k)$ of order $\ell^e$ that depends only on the coset
$K_\ell(N)g_\ell$, as one can easily check. The point $x_a\in E_{\psi_a}[N](\bar k)$ is the one
whose $\ell$--th primary component is $x_{a,\ell}$, and depends only on $a$, and not on $g$.


To construct the invariant $k$--form $\omega_a$, we use the fact that for any object $X$
of $\mathcal{C}_{-p}$ there is a canonical identification of $k$--vector spaces (cf. \cite{Fo},
II, Prop. $4.3$)
\begin{equation}\label{cotan}T_p(X)/FT_p(X)={\mathbf{t}_0(X)}^*.
\end{equation}
Recall that the Dieudonn\'e module $T_p(E_{\psi_a})$ is identified, via $\psi_{a,p}$, to a
$\mathcal{A}$--submodule $M_{\psi_{a,p}}$ of $T_p(E)$, moreover the component $g_p\in D_p^*$ at $p$ of $g$
defines a generator $\iota_p(g_p)$ of $\ho_\mathcal{A}(T_p(E),M_{\psi_{a,p}})$ as a right ideal of
$\enn_\mathcal{A}(T_p(E))$. 
%

We have a commutative diagram of morphisms of $\mathcal{A}$--modules
$$\xymatrix{T_p(E)&\ar@{_(->}[l] M_{\psi_{a,p}}&T_p(E)\ar[l]_{\sim}\ar@(ur,ul)[ll]_{\iota_p(g_p)}\ar[dl]_{\sim}\\
&T_p(E_{\psi_a}).\ar[lu]^{\psi_{a,p}}\ar[u]^\sim&\\}$$
The diagonal isomorphism to the right of the diagram is induced by the composition of $\iota_p(g_p)$ corestricted to $M_{\psi_{a,p}}$,
followed by the inverse of the corestriction of $\psi_{p,a}$ to the same lattice of $V_p(E)$. Notice that it is not
independent on the choice of $g_p$. From this isomorphism, using (\ref{cotan}), we can deduce an isomorphism
of $k$--vector spaces
$$\xi_a:{\mathbf{t}_0(E)}^*\xrightarrow{\sim}{\mathbf{t}_0(E_{\psi_a})}^*$$
that depend only on the coset $R_p^*(1)g_p$, and not on the choice of $g_p$, as one can easily check.
Setting $\omega_a$ to be equal to $\xi_a(\omega_0)$ completes the construction of the triple
$(\psi_a,\omega_a,x_a)$ attached to $a=K(1,N)g\in K(1,N)\backslash\hat R\cap\hat D^*$.

We are left with showing that, for $a,$ $b\in K(1,N)\backslash\hat R\cap\hat D^*$, the triples
$(\psi_a,\omega_a,x_a)$ and $(\psi_b,\omega_b,x_b)$ are isomorphic if and only if $a=b\lambda$,
for some $\lambda\in D^*$. One way to proceed is to observe that in order for the two triples to be isomorphic
certainly we must have that there exists an isomorphism $u:E_{\psi_a}\rightarrow E_{\psi_b}$, equivalently, by Proposition
\ref{tau}, if $\bar a,$ $\bar b$ denote the elements of $\hat R^*\backslash\hat R\cap\hat D^*$ deduced from, respectively, $a$
and $b$, then there must be $\lambda\in D^*$ such that $\bar a=\bar b\lambda$. We leave to the reader the task of
showing, with the help of section \ref{subiso}, that the isomorphism $u$ induces an isomorphism between
$(\psi_a,\omega_a,x_a)$ and $(\psi_b,\omega_b,x_b)$ if and only if $a=b\lambda$.
\end{proof}

\begin{rem} If $d$ is an integer $\geq 1$, the natural map
$K(1,Nd)\backslash\hat D^*/D^*\rightarrow K(1,N)\backslash\hat D^*/D^*$ corresponds, under the identifications
$\tau_{Nd}$ and $\tau_N$ to the map sending a given triple $(\psi,\omega, x)\in\Sigma(1,Nd)$ to $(\psi,\omega, dx)\in\Sigma(1,N)$.
\end{rem}

\begin{rem}\label{dbi} Once for each prime $\ell\neq p$ a basis for $T_\ell(E)$ as in Remark \ref{defineun} is chosen, we have that
the canonical involution $\gamma$ of $D$ defines a bijection
$$K(1,N)\backslash\hat D^*/D^*\xrightarrow{\sim}\Omega(1,N),$$
where $\Omega(1,N)$ is the double coset space introduced in section \ref{autom}.  The effect of the main involution of
$D_p$ on the multiplicative group of the residue field is the Frobenus automorphism. This explains the occurrence of the twist for
the action of $k^*$ as given in Theorem \ref{corr3}, which is just a restatement of Theorem \ref{corr2}.
\end{rem}

\subsection{Hecke operators}\label{operiamo} For an integer $N\geq 1$ not divisible by $p$, the space $S(1,N)$
defined in section \ref{autom} is identified, using Remark \ref{dbi} following Theorem \ref{corr3}, with
the space of complex valued functions on $\Sigma(1,N)$.
Our task is now to investigate a few basic properties of the Hecke operators on $S(1,N)$.

\begin{defi} Let $\ell$ be a prime $\neq p$. The $\ell$--th Hecke operator $T_\ell$ on $S(1,N)$ is defined
as
$$T_\ell (f)(\psi,\omega,x)=\sum_{\varphi}f((\psi',\omega',x')),$$
where the sum ranges through all the degree $\ell$ isogenies $\varphi:E_\psi\rightarrow E_{\psi'}$
such that $\varphi(x)$ has order $N$ in $E_\psi'[N]$, and where remaining entries of the triple $(\psi',\omega',x')$ are
defined by $\varphi^*(\omega')=\omega$, $\varphi(x)=x'$.
\end{defi}

\begin{prop} If $\ell\nmid pN$, then the Hecke operator $T_\ell$ acting on $S(1,N)$ is semi--simple.
\end{prop}
\begin{proof}\label{heckeisss} Consider the real inner product on $S(1,N)$ for which the basis given by the characteristic functions
of the singletons of $\Sigma(1,N)$ is orthonormal. A standard application of the existence of the dual isogeny shows that
$T_\ell$, for $\ell\nmid pN$ commutes with its adjoint and is therefore semi--simple (cf. \cite{Gr}, Prop. 2.7).
\end{proof}

For $\ell\neq p$, consider the decomposition into right cosets of the double coset
\begin{equation}U_\ell(N)^\iota\begin{pmatrix}
1&0\\
0&\ell\\
\end{pmatrix}U_\ell(N)^\iota=\bigsqcup_{i}U_\ell(N)^\iota\alpha_i.
\end{equation}
The number of left cosets in the above decomposition is $\ell+1$ if $\ell\nmid pN$, and $\ell$ if $\ell| N$.
Using the identification
$$\Sigma(1,N)=\left(R_p^*(1)\times\prod_{\ell\neq p}U_\ell(N)^\iota\right)\backslash\hat D^*/D^*,$$
one can show that the operator $T_\ell$ is described by:

\begin{prop} For every $\ell\neq p$ we have
$$T_\ell(f)(x)=\sum_{i}f(\alpha_i x).$$
\end{prop}

After applying the main involution, we now see that the definition of $T_\ell$ given in section \ref{autom} corresponds to
that given in this section under the bijection of Remark \ref{dbi}.


\begin{thebibliography}{I12}

\bibitem{BGP} Baker, Gonz‡lez--JimŽnez, Gonz‡lez, Poonen,
{\it Finiteness theorems for modular curves of genus at least 2}, Amer. J. Math. 127 (2005), 1325--1387.

\bibitem{Ca} W. Casselman, {\it On some results of Atkin and Lehner}, Math. Ann. 201 (1973), 301--314.

\bibitem{Del} P. Deligne, {\it Formes modulaires et repr\'esentations de $\textrm{GL}_2$},
Modular Functions of One Variable II, Lecture Notes in Math. {\bf 349}, Springer--Verlag, (1973), 55--105.

\bibitem{De} M. Demazure, {\it Lectures on $p$--divisible groups}, Lecture Notes in Mathematics, Vol. 302.
Springer--Verlag, Berlin--New York, 1972.

\bibitem{Fo} J.--M. Fontaine, {\it Groupes $p$--divisibles sur les corps locaux}, Ast\'erisque 47-48,
Soci\'et\'e Math\'ematique de France, (1977).

\bibitem{GG} I. M. Gelfand, M. I. Graev, I. Piatetski--Shapiro, {\it Representation theory and automorphic functions},
W. B. Saunders Co., 1969.

\bibitem{Gr} B. H. Gross, {\it Heights and special values of L-series}, CMS Proceedings, Vol. 7, AMS (1986), 115--187.

\bibitem{Gh} A. Ghitza, {\it Hecke eigenvalues of Siegel modular forms (mod p) and of algebraic modular forms},
Journal of Number Theory 106, no. 2 (2004), 345--384. 

\bibitem{Gh2} A. Ghitza, {\it Siegel modular forms (mod p) and algebraic modular forms}, PhD thesis.

\bibitem{JL} H. Jacquet, R. P. Langlands {\it Automorphic forms on GL(2)}, LNM 114, Springer--Verlag 1970.

\bibitem{Ro} J. Rogawski, {\it Appendix: Modular forms, the Ramanujan conjecture, and the Jacquet--Langlands correspondence},\\
available at http://www.math.ucla.edu/$\sim$jonr/eprints.html.

\bibitem{Se}
J.--P. Serre,
{\it Two letters on Quaternions and Modular Forms (mod $p$)}. With introduction,
appendix and references by R. Livn\'e. Israel J. Math.  {\bf 95}, 281--299 (1996).

\bibitem{Si} J. H. Silverman, {\it The arithmetic of elliptic cuurves}, Graduate texts in Mathematics 106,
1986 Spinger-Verlag, New York.

\bibitem{T} J. Tate, {\it Classes d'isog\'enie des vari\'et\'es ab\'eliennes sur un corps fini},
S\'em. Bourbaki 21e ann\'ee, 1968/69, no 352.

\bibitem{T2} J. Tate, {\it Endomorphisms of abelian varieties over finite fields}, Inventiones Math.
$2$, 134--144 (1966).

\bibitem{Vi} M.--F. Vign\'eras, {\it Artihm\'etiques des alg\`ebres de quaternions}, Lecture Notes in Mathematics, Vol. 800.
Springer--Verlag, Berlin--New York, 1980.

\bibitem{W} W.C. Waterhouse, {\it Abelian varieties over finite fields}, Ann. scient. \'Ec. Norm. Sup.,
$4^e$ s\'erie, t. 2, 1969, 521--560.

\bibitem{WM} W.C. Waterhouse, J.S. Milne, {\it Abelian varieties over finite fields}, 1969 Number Theory Institute
(Proc. Sympos. Pure Math., Vol. XX, State Univ. New York, Stony Brook, N.Y., 1969), pp. 53--64. Amer. Math. Soc.,
Providence, R.I., 1971

\end{thebibliography}
\end{document}